\documentclass[12pt]{amsart}
\usepackage{geometry} 
\usepackage{enumitem} 
\usepackage{amssymb,amsmath,mathrsfs,enumitem,upgreek,mathtools,diagbox}
\usepackage{microtype}
\usepackage{tikz}

\newcommand{\CH}{{\mathsf{CH}}}	
\DeclareMathOperator{\Int}{Int} 		
\DeclareMathOperator{\Cl}{Cl} 		
\newcommand{\FORALL}[1]{\forall {#1} \, }
\newcommand{\EXISTS}[1]{\exists {#1} \, }
\newcommand{\N}{\mathbb{N}}	
\newcommand{\R}{\mathbb{R}}	
\newcommand{\ZFC}{\mathsf{ZFC}}		
\newcommand{\Pow}{\mathscr{P}}		
\newcommand{\IMPLIES}{\Rightarrow }
\newcommand{\IFF}{\Leftrightarrow}
\newcommand{\AND}{\mathbin{\, \wedge \,}}
\newcommand{\markdef}[1]{\textbf{#1}} 
\newcommand{\Mid}{\boldsymbol\mid}			
\newcommand{\set}[1]{\mathopen \{ {#1} \mathclose \}} 
\newcommand{\setof}[2]{\mathopen \{{#1}\Mid{#2} \mathclose\}} 
\newcommand{\setLR}[1]{\left \{ {#1} \right \}} 
\newcommand{\card}[1]{\mathopen | #1 \mathclose |}		
\newcommand{\cardLR}[1]{ \left | #1 \right |}		
\newcommand{\On}{\mathord{\mathrm{Ord}}}		
\DeclareMathOperator{\Span}{span}
\renewcommand{\restriction}{\mathop{\upharpoonright}}

\theoremstyle{plain}
\newtheorem{theorem}{Theorem}[section]
\newtheorem{conjecture}[theorem]{Conjecture}
\newtheorem{proposition}[theorem]{Proposition}
\newtheorem{lemma}[theorem]{Lemma}
\newtheorem{corollary}[theorem]{Corollary}
\newtheorem{claim}{Claim}[theorem]

\theoremstyle{definition}
\newtheorem{definition}[theorem]{Definition}
\newtheorem{notation}[theorem]{Notation}
\newtheorem{problem}[theorem]{Problem}
\theoremstyle{remark}
\newtheorem{remark}[theorem]{Remark}
\newtheorem{remarks}[theorem]{Remarks}
\newtheorem{example}[theorem]{Example}

\usepackage[style=alphabetic,sorting=nyt,backend=bibtex8]{biblatex}
\bibliography{sprays.bib}

\usepackage{hyperref}

\title{How many sprays cover the space?}
\author{Alessandro Andretta}
\address{Università degli Studi di Torino, Dipartimento di Matematica ``G. Peano", Via Carlo Alberto 10, 10123 Torino, Italy}
\email{alessandro.andretta@unito.it}
\author{Ivan Izmestiev}
\address{TU Wien,
Institute of Discrete Mathematics and Geometry,
Wiedner Hauptstra{\ss}e 8-10/104,
1040 Vienna, Austria}
\email{izmestiev@dmg.tuwien.ac.at}
\date{\today}

\subjclass[2020]{03E50, 51N20}
\keywords{Continuum hypothesis, sprays}

\begin{document}
\maketitle

\begin{abstract}
For all \( d \geq 3 \) we show that the cardinality of \( \R \) is at most \( \aleph_n \) if and only if \( \R^d \) can be covered with \( ( n + 1 ) ( d - 1 ) + 1 \) sprays whose centers are in general position in a hyperplane.
This extends previous results by Schmerl when \( d = 2 \).
\end{abstract}

\section{Introduction}
A spray in the plane is a subset of \( \R^2 \) together with a distinguished point (called center) such that all circles centered in that point have finite intersection with the spray.
These objects were introduced in~\cite{Schmerl:2003uq}, where it is observed that the continuum hypothesis implies that the plane can be covered with three sprays.
(The continuum hypothesis, \( \CH \) from now on, asserts that every infinite subset of \( \R \) is either countable or else in bijection with \( \R \); by work of Gödel and Cohen \( \CH \) can neither be disproved nor proved from \( \ZFC \), the Zermelo-Frænkel set theory with the axiom of choice.)
In~\cite{Vega:2009xy} it is shown that \( \CH \) is equivalent to the plane being covered with three sprays with collinear centers---in fact the statement ``the cardinality of \( \R \) is at most \( \aleph_n \)'' (in symbols: \( 2^{\aleph_0} \leq \aleph_n \)) is equivalent to \( \R^2 \) being covered with \( n + 2 \) sprays with collinear centers~\cite{Schmerl:2010nr}.
Collinearity is essential, since \( \ZFC \) proves that the plane can be covered with three sprays centered in arbitrary non-collinear points~\cite{Vega:2009xy,Schmerl:2010nr}.

The notion of spray can be extended to \( \R^d \) for any \( d \geq 3 \), with \( ( d - 1 ) \)-dimensional spheres in place of \( 1 \)-dimensional spheres, i.e. circles, and the natural question is the relation (if any) between the size of the continuum and the number of sprays needed to cover \( \R^d \).
By~\cite[Theorem 2]{Erdos:1994yq}, \( 2^{\aleph_0} \leq \aleph_n \) implies that \( \R^3 \) can be covered with \( 2 n + 3 \) sprays such that their centers are coplanar, and no three of them are collinear.
In particular, \( \CH \) implies that given five coplanar points such that no three of them are collinear, then there are sprays centered around them that cover \( \R^3 \).
By~\cite{Schmerl:2012aa} \( \R^3 \) cannot be covered with three sprays, but the problem whether the space can be covered with four sprays remained open.

We prove that \( 2^{\aleph_0} \leq \aleph_n \) is equivalent to \( \R^3 \) being covered with \( 2 n + 3 \) sprays whose centers are coplanar, and no three of them are collinear.
Moreover \( 2 n + 3 \) is optimal---in particular \( \R^3 \) cannot be covered with four sprays with coplanar centers.
In fact we prove a similar result for \( \R^d \) with \( d > 3 \), namely: \( 2^{\aleph_0} \leq \aleph_n \) is equivalent to \( \R^d \) being covered with \( ( n + 1 ) ( d - 1 ) + 1 \)-many sprays whose centers lie on a hyperplane \( H \), and the affine span of any \( d \) of them is \( H \) (Theorem~\ref{th:spraysinRd}).
Again the number \( ( n + 1 ) ( d - 1 ) + 1 \) of sprays is optimal.
Finally we show that, irrespective of the size of the continuum, \( \R^d \) can be covered with countably many sprays whose centers lie on a hyperplane \( H \), and the affine span of any \( d \) of them is \( H \) (Theorem~\ref{th:unionofcountablymanysprays}).

We do not know if \( \R^d \) can be covered with \( d + 1 \)-many sprays such that the affine span of their centers is \( \R^d \), even in the case \( d = 3 \).
In other words: can \( \R^3 \) be covered with four sprays whose centers are not coplanar, i.e. form a tetrahedron?
We suspect that the answer is affirmative and that it can be proved in \( \ZFC \).
If true, this would be analogous to the fact that \( \R^2 \) can be covered with three sprays with non-collinear centers.

The results in this paper lie on the interface between set theory and the geometry of euclidean spaces.
The archetypal results in this area are the classical theorems of Sierpiński---see Theorem~\ref{th:Sierpinski} below and the beginning of Section~\ref{sec:Hyperplane-sections} for some context and references.
Most of these results deal with linear objects like lines, or hyperplanes in \( \R^d \), while sprays are essentially non-linear objects.
We extend a construction in~\cite{Schmerl:2010nr} from \( \R^2 \) to \( \R^d \), transforming the quadratic problem of covering the space with sprays to the linear problem of covering the space with sets having finite intersections with certain families of hyperplanes.
This latter problem has been studied before~\cite{Bagemihl:1959te,Erdos:1994yq,Simms:1997aa}.
Extending these earlier results we are able to prove our results on sprays.
 
The paper is organized as follows.
After recalling the notations and the basic notions that will be used throughout the paper, we show in Section~\ref{sec:Transforming-sprays-into-linear-objects} how to transform a covering of \( \R^d \) with sprays with centers on a given hyperplane into a covering of (an open subset of) \( \R^d \) with small intersection with certain families of hyperplanes.
Section~\ref{sec:Hyperplane-sections} is devoted to studying the following problem: given distinct, non-zero vectors \( \mathbf{u}_1 , \dots , \mathbf{u}_k \) in \( \R^d \), are there \( A_1 , \dots , A_k \) covering \( \R^d \) such that every plane orthogonal to \( \mathbf{u}_i \) has finite intersection with \( A_i \)? 
It turns out that this problem is closely related to the size of the continuum, and by the results in Section~\ref{sec:Transforming-sprays-into-linear-objects} it is equivalent to the existence of sprays \( X_1 , \dots , X_k \) with centers on a given hyperplane \( H \), and covering \( \R^d \), as the (directions of the) \( \mathbf{u}_i \)s correspond to the position of the centers of the \( X_i \)s on \( H \).
Finally in Section~\ref{sec:The-main-results} we prove the results about the existence of sprays covering \( \R^d \) and the size of the continuum.

In order to make the present paper accessible to a wider audience, many basic definitions in set theory and geometry are summarized in Section~\ref{subsec:Notation}.
The reader who is proficient with these matters can safely skip it.

\section{Notation and preliminary results}\label{sec:Notation-and-preliminary-results}
\subsection{Notation}\label{subsec:Notation}

\subsubsection{Set theory}
We work in \( \ZFC \)---this is the standard framework that most mathematicians (implicitly, or explicitly) adopt to prove theorems.
Our notation is standard: \( \Pow ( X ) \) is the power-set of \( X \), the pointwise image of \( Z \subseteq X \) via some \( f \colon X \to Y \) is \( f [ Z ] \coloneqq \setof{ f ( z ) }{ z \in Z } \).

\markdef{Ordinals} are special sets that are well-ordered by \( \in \), the membership relation, and every well-ordered set is isomorphic to a unique ordinal.
Lower-case Greek letters range over the ordinals, \( \On \) is the collection of all ordinals, and for \( \alpha , \beta \in \On \) we have \( \alpha < \beta \) if and only if \( \alpha \in \beta \).
Any natural number can be construed as an ordinal, letting \( 0 \) be \( \emptyset \) the empty set, \( 1 \) be the set \( \setLR{0} \), and more generally \( n + 1 \coloneqq n \cup \setLR{ n } \).
The least infinite ordinal \( \omega \) is the set \( \N = \setLR{ 0 , 1 , \dots , n , \dots } \).
A \markdef{cardinal} is an ordinal that is not in bijection with a smaller ordinal---every natural number is a cardinal, and so is \( \omega \).
Infinite cardinals are enumerated by the function \( \aleph \) using ordinals as indexes; so \( \aleph_0 \) is \( \omega \), the smallest infinite cardinal, \( \aleph_1 \) is the smallest cardinal bigger than \( \aleph_0 \), and more generally \( \aleph_{n + 1} \) is the smallest cardinal bigger than \( \aleph_n \).
As \( \omega \) is the smallest ordinal bigger than any finite ordinal (i.e. natural number), \( \aleph_ \omega \) is the smallest cardinal bigger than any \( \aleph_n \).
As there is no reason to stop there we can take larger ordinals such as \( \omega + 1 < \omega + 2 < \dots \) and construct even larger cardinals \( \aleph_{ \omega + 1} < \aleph_{ \omega + 2 } < \dots \).
And so on.

Using the axiom of choice, every set \( X \) can be well-ordered, so it is in bijection with some ordinal, and hence with a unique cardinal: this cardinal is denoted by \( \card{X} \) and it is called the \markdef{cardinality} of \( X \).
A set \( X \) is finite if it is in bijection with a natural number i.e. \( \card{X} < \aleph_0 \); otherwise it is infinite i.e. \( \aleph_0 \leq \card{X} \).
We say \( X \) is \markdef{countable} if either \( \card{X} < \aleph_0 \), or else \( \card{X} = \aleph_0 \)---this can be written in a compact way as \( \card{X} \leq \aleph_0 \) or as \( \card{X} < \aleph_1 \).
A set \( X \) is \markdef{uncountable} if it is not countable, that is \( \aleph_1 \leq \card{X} \).
The countable union of countable sets is countable. 
More generally, for any ordinal \( \delta \), if \( ( X_i )_{ i \in I } \) is a family of sets, each of cardinality \( \leq \aleph_ \delta \) and \( \card{I} \leq \aleph_ \delta \), then \( \card{ \bigcup_{i \in I} X_i } \leq \aleph_ \delta \).

The set \( \R \) is in bijection with \( \setof{ f}{ f \colon \omega \to \setLR{ 0 , 1 }} \), and for this reason the cardinality of \( \R \) is often denoted with \( 2^{\aleph_0} \).
By Cantor's theorem \( \aleph_1 \leq \card{\R} \), and Cantor's continuum hypothesis \( \CH \) asserts that the inequality can be replaced with an equality.
In other words \( \CH \) asserts that every subset of the real line is either countable or else it is in bijection with \( \R \).
Since \( \aleph_1 \leq 2^{\aleph_0} \) one can state \( \CH \) as \( 2^{\aleph_0} \leq \aleph_1 \).
By Cohen's results, it is consistent that the cardinality of \( \R \) be any \( \aleph_{ n + 1 } \), or even larger cardinals, like \( \aleph_{ \omega + n + 1 } \).
(But \( 2^{\aleph_0} \) cannot be \( \aleph_ \omega \) by König's theorem.)

\subsubsection{Geometry}
The standard basis of \( \R^d \) is denoted by \( \mathbf{e}_1 , \dots , \mathbf{e}_d \).
The following notation for hyperplanes will be used throughout the paper.

\begin{notation}\label{ntn:hyperplane}
Given a vector \( \mathbf{u} \in \R^d \setminus \setLR{ \mathbf{0} } \) and \( \mathbf{p} \in \R^d \) let 
\[
H_{\mathbf{u}} ( \mathbf{p} ) = \mathbf{p} + \setof{ \mathbf{v} }{ \mathbf{v} \boldsymbol{\cdot} \mathbf{u} = 0 }
\]
be the hyperplane of \( \R^d \) orthogonal to \( \mathbf{u} \) passing through \( \mathbf{p} \), where \( \boldsymbol{\cdot} \) is the (standard) inner product.
We also let \( H_i ( x ) = \setof{ ( p_1 , \dots , p_d ) \in \R^d }{ p_i = x } \).
\end{notation}

An \markdef{affine subspace} \( E \) of \( \R^d \) is a translate of some vector subspace of  \( \R^d \), that is a set of the form \( \mathbf{p} + V \) where \( V \) is a vector subspace of \( \R^d \); the vector subspace \( V \) is unique, and it is the vector space associated to \( E \), while \( \mathbf{p} + V = \mathbf{p}' + V \) if and only if \( \mathbf{p} - \mathbf{p}' \in V \).
In this paper, a (finite dimensional) \markdef{affine space} is an affine subspace of some \( \R^d \).
Elements of an affine space are called points, and since every vector subspace \( V \) of \( \R^d \) is also an affine space, we can refer to elements of it as points or vectors, depending if we privilege the affine or vector space structure. 
If \( E \) is an affine space of \( \R^d \) and \( \mathbf{p} , \mathbf{q} \in E \) then \( \mathbf{q} - \mathbf{p} \) belongs to \( V \), the underlying vector space of \( E \), and if \( \mathbf{v} \in V \) then \( \mathbf{p} + \mathbf{v} \) belongs to \( E \).
The dimension of an affine space is, by definition, the dimension of the associated vector space.
Given an affine space \( E \) with associated vector space \( V \), the \markdef{affine envelope} or \markdef{affine span} of a non-empty set \( S \subseteq E \) is \( \mathbf{p} + \Span \setof{ \mathbf{q} - \mathbf{p} }{ \mathbf{q} \in S } \) where \( \mathbf{p} \in S \) and \( \Span X \) is the smallest vector subspace of \( V \) containing \( X \subseteq V \).
It is easy to check that the definition does not depend on the point \( \mathbf{p} \), and that the affine span of \( S \) is the intersection of all affine subspaces of \( E \) containing \( S \).
Two affine subspaces \( E \) and \( F \) of \( \R^d \) are \markdef{complementary} if their associated vector spaces \( V \) and \( W \) are complementary, that is \( \R^d = V + W \) and \( V \cap W = \setLR{\mathbf{0}} \).
The intersection of two complementary spaces \( E , F \) is a single point.
In fact if \( \mathbf{p} , \mathbf{q} \in E \cap F \), then \( \mathbf{p} - \mathbf{q} \in V \cap W = \setLR{\mathbf{0}} \), so \( \mathbf{p} = \mathbf{q} \), that is \( E \cap F \) has at most one point.
If \( E = \mathbf{p} + V \) then \( \mathbf{p} \) can be written in a unique way as \( \mathbf{v} + \bar{\mathbf{w}} \) with \( \mathbf{v} \in V \) and \( \bar{\mathbf{w}} \in W \), and \( E = \bar{\mathbf{w}} + V \); similarly \( F = \bar{\mathbf{v}} + W \) for a unique \( \bar{\mathbf{v}} \in V \).
Therefore \( \bar{\mathbf{v}} + \bar{\mathbf{w}} \) is the unique element of \( E \cap F \). 

The inner product \( \boldsymbol{\cdot} \) on \( \R^d \) yields the Euclidean norm \( \| \cdot \| \), and hence we have a distance on \( \R^d \) and all its affine subspaces.
In particular, if \( E \) is an affine space then \( \| \mathbf{p} - \mathbf{q} \| \) is the distance between \( \mathbf{p} , \mathbf{q} \in E \).
Two vectors \( \mathbf{v} ,\mathbf{w} \) of \( \R^d \) are \markdef{orthogonal} if \( \mathbf{v} \boldsymbol{\cdot} \mathbf{w} = 0 \).
If \( V , W \) are complementary subspaces of \( \R^d \), then \( V \) and \( W \) are orthogonal if \( \FORALL{\mathbf{v} \in V} \FORALL{\mathbf{w} \in W} ( \mathbf{v} \boldsymbol{\cdot} \mathbf{w} = 0 ) \); in this case we call one of the two subspaces the \markdef{orthogonal complement} of the other.
Two complementary affine subspaces \( E \) and \( F \) of \( \R^d \) are orthogonal if their underlying vector spaces are orthogonal.

Let \( E \) be an affine subspace of \( \R^d \).
The \markdef{sphere in \( E\) with center \( \mathbf{c} \in E \) and radius \( r \in \R \)} is the set
\[
 \mathbb{S} ( E ; \mathbf{c} , r ) \coloneqq \setof { \mathbf{x} \in E }{ \| \mathbf{x} - \mathbf{c} \| = r } .
\]
We convene that \( \mathbb{S} ( E ; \mathbf{c} , r ) \) is empty if \( r < 0 \), and it is the singleton \( \setLR{ \mathbf{c} } \) when \( r = 0 \).
Observe that if \( r > 0 \) then \( \mathbb{S} ( E ; \mathbf{c} , r ) \) has cardinality \( 2 \) when \( \dim ( E ) = 1 \); if \( \dim ( E ) \geq 2 \), then \( \mathbb{S} ( E ; \mathbf{c} , r ) \) has cardinality \( 2^{\aleph_0} \).
Whenever the ambient space (i.e. \( \R^d \), \( E \), \ldots) is clear we will simply write \( \mathbb{S} ( \mathbf{c} , r ) \).
A \markdef{\( ( k - 1 ) \)-dimensional sphere} is a sphere in an affine subspace \( E \) of some \( \R^d \) with \( \dim E = k \).

\subsection{Families with finite mesh}
\begin{definition}\label{def:finitemesh}
Let \( N , d \geq 2 \) be natural numbers and suppose that \( \mathcal{H}_i \subseteq \Pow ( \R^d ) \) with \( 1 \leq i \leq N \) are pairwise disjoint.
The sequence \( ( \mathcal{H}_i )_{ i = 1}^N \) 
\begin{itemize}
\item
is \markdef{locally finite} if there is a positive integer \( s \) such that for any \( s \) distinct points in \( \R^d \) there are at most finitely many sets in \( \bigcup_{ i = 1 }^N \mathcal{H}_i \) to which these points belong;
\item
has \markdef{mesh} \( r \), if \( r \geq 2 \) is the least integer such that the intersection of \( r \) sets belonging to distinct \( \mathcal{H}_i \)s is a finite set.
\end{itemize}
\end{definition}
All sequences \( ( \mathcal{H}_i )_{ i = 1 }^N \) we consider are locally finite.
The reason for Definition~\ref{def:finitemesh} is the following very general result by Erd\H{o}s, Jackson, and Mauldin~\cite[Theorem 2]{Erdos:1994yq}:

\begin{theorem}\label{th:EJM2}
For all \( d \geq 2 \), \( n \geq 0 \), and \( \delta = 0 , 1 \) the following are equivalent:
\begin{enumerate}[label={\upshape (\alph*)}]
\item\label{th:EJM2-a}
\( 2^{\aleph_0} \leq \aleph_{ \delta + n } \);
\item\label{th:EJM2-b}
for any \( r \geq 2 \), letting \( N = ( n + 1 ) ( r - 1 ) + 1 \), and for any sequence of pairwise disjoint \( \mathcal{H}_i \subseteq \Pow ( \R^d ) \) with \( 1 \leq i \leq N \) locally finite and of mesh \( r \), there are \( A_1 , \dots , A_N \) covering \( \R^d \) such that \( \FORALL{1 \leq i \leq N} \FORALL{ H \in \mathcal{H}_i } ( \card{ H \cap A_i } < \aleph_{ \delta} ) \).
\end{enumerate}
\end{theorem}

The theorem remains true if \( \delta \) is taken to be an ordinal, not just \( 0 \) or \( 1 \).
The main use of Theorem~\ref{th:EJM2} in this paper is the direction \ref{th:EJM2-a}\( \IMPLIES \)\ref{th:EJM2-b}.
Given \( ( \mathcal{H}_i )_{ i = 1 }^N \) locally finite of mesh \( r \) in \( \R^d \), then \( \card{\R} \leq \aleph_{ \delta + n } \) implies that there are \( A_1 , \dots , A_N \subseteq \R^d \) covering \( \R^d \) such that for all \( H \in \mathcal{H}_i \), the set \( H \cap A_i \) is \emph{finite} if \( \delta = 0 \) or \emph{countable} if \( \delta = 1 \).
Actually, the ``forward implication'' (i.e.~from a bound on the size of the continuum, to the existence of specific subsets of the space) in many papers~\cite{Bagemihl:1959te,Davies:1962ly,Davies:1963yu,Davies:1963zr,Bagemihl:1968up,Komjath:2001kq} follows from Theorem~\ref{th:EJM2}.

\subsection{Points and vectors in general position}

\begin{definition}\label{def:well-spread}
A set of vectors of \( \R^d \) is in \markdef{general position} if any subset of size \( \leq d \) is linearly independent---in other words: the vectors are as linearly independent as possible. \end{definition}

\begin{example}\label{xmp:well-spread}
If \( \mathbf{u}_1 , \dots , \mathbf{u}_N \) are vectors in general position in \( \R^d \), and \( \mathcal{H}_i \coloneqq \setof{ H_{ \mathbf{u}_i } ( \mathbf{p} ) }{ \mathbf{p} \in \R^d } \) is the family of all hyperplanes orthogonal to \( \mathbf{u}_i \), then \( ( \mathcal{H}_i )_{ i = 1 }^N \) is locally finite, and of mesh \( d \).
\end{example}

By Example~\ref{xmp:well-spread} and Theorem~\ref{th:EJM2}, if \( \mathbf{u}_1 , \dots , \mathbf{u}_{ ( n + 1 ) ( d - 1 ) + 1 } \) are vectors in general position in \( \R^d \), then

\begin{itemize}
\item
\( 2^{\aleph_0} \leq \aleph_{n} \) implies that there are \( A_1 , \dots , A_{ ( n + 1 ) ( d - 1 ) + 1 } \) covering \( \R^d \) such that \( H_{\mathbf{u}_i} ( \mathbf{p} ) \cap A_i \) is finite for every \( \mathbf{p} \in \R^d \) and \( 1 \leq i \leq ( n + 1 ) ( d - 1 ) + 1 \);
\item
\( 2^{\aleph_0} \leq \aleph_{ n + 1} \) implies that there are \( A_1 , \dots , A_{ ( n + 1 ) ( d - 1 ) + 1 } \) covering \( \R^d \) such that \( H_{\mathbf{u}_i} ( \mathbf{p} ) \cap A_i \) is countable, for every \( \mathbf{p} \in \R^d \) and \( 1 \leq i \leq ( n + 1 ) ( d - 1 ) + 1 \).
\end{itemize}

Therefore \( \CH \) implies that if \( \mathbf{u}_1 , \dots , \mathbf{u}_d \) are a basis of \( \R^d \), then there are \( A_1 , \dots , A_ d \) covering \( \R^d \) such that every hyperplane orthogonal to \( \mathbf{u}_i \) has countable intersection with \( A_i \).

In fact, the implications above are actually equivalences~\cite[pp.~71, 327--328]{Komjath:2006fj}.

\begin{theorem}[Sierpiński]\label{th:Sierpinski}
\( \CH \) is equivalent to either one of the following:
\begin{enumerate}[label={\upshape (\alph*)}]
\item\label{th:Sierpinski-a}
there are \( A_1 , A_2 \) covering the plane such that every vertical line has countable intersection with \( A_1 \) and every horizontal line has countable intersection with \( A_2 \),
\item\label{th:Sierpinski-b}
there are \( A_1 , A_2 , A_3 \) covering \( \R^3 \) such that every plane orthogonal to \( \mathbf{e}_i \) has countable intersection with \( A_i \).
\end{enumerate}
\end{theorem}

Definition~\ref{def:well-spread} can be extended to points in affine spaces.

\begin{definition}\label{def:generalposition}
Let \( S \) be a set of points of an affine space \( E \) of dimension \( d \). 
We say that \( S \) is 
\begin{itemize}
\item
in \markdef{general position in \( E \)} if the affine span of any of its subset of size \( k + 1 \) has dimension \( k \), for all \( k \leq d \),
\item
 \markdef{well-placed in} \( E \) if \( S \subseteq H \) for some hyperplane \( H \), and \( S \) is in general position in \( H \).
\end{itemize}
\end{definition}

\begin{remarks}\label{rmks:well-placed}
Let \( E \) be an affine space of dimension \( d \geq 2 \) and let \( S \subseteq E \).
\begin{enumerate}[label={\upshape (\alph*)}]
\item\label{rmks:well-placed-a}
If \( S \) is in general position (well-placed), and \( S' \subseteq S \) then \( S' \) is in general position (well-placed).
In other words, the notions of being in general position/well-placed are downward persistent with respect to inclusion.
\item
Suppose \( S \) is well-placed: 
\begin{itemize}
\item
if \( \card{S} \geq d \) then the hyperplane in the definition is unique, being the affine span of any subset of size \( d \);
\item
if \( \card{S} \leq d \) then \( S \) is in general position in \( E \).
\end{itemize}
\item
Suppose \( H \) is a hyperplane of \( E \), and \( S \subseteq H \).
Then \( S \) is well-placed in \( E \) if and only if \( S \) is in general position in \( H \).
\item
\( S \) is in general position in \( E \) if and only if the set of vectors \( \setof{ \mathbf{q} - \mathbf{p} }{ \mathbf{q} \in S \setminus \setLR{\mathbf{p} } } \) is in general position in \( \R^d \), for all \( \mathbf{p} \in S \).
\end{enumerate}
\end{remarks}

Every set of points in \( \R \) is in general position, a set of points in \( \R^2 \) is in general position if no three of them are aligned, a set of points in \( \R^3 \) is in general position if no four of them are coplanar, and so on. 
A set of points in \( \R^2 \) is well-placed if they are collinear, a set of points in \( \R^3 \) is well-placed if they are coplanar, and no three of them are collinear, a set of points in \( \R^4 \) is well-placed if they belong to the same \( 3 \)-dimensional affine subspace, and no four of them are coplanar, and so on.

Finally observe that for any affine space \( E \) of dimension \( d \geq 2 \) and any \( S \subseteq E \)
\begin{equation}\label{eq:properlyplaced}
\begin{multlined}
 \bigl ( S \text{ is general position in \( E \), or \( S \) is well-placed in } E \bigr ) \IMPLIES {}\qquad
 \\
 \shoveright{\qquad\FORALL{ k < d}\bigl ( \text{the affine span of any subset of \( S \) of size } k + 1}
 \\
 \qquad\text{has dimension } k \bigr ) .
\end{multlined}
\end{equation}

\begin{lemma}\label{lem:projectiogeneralposition}
Suppose \( \mathbf{c}_1 , \dots , \mathbf{c}_{ k + 1 } \) are distinct points in general position in \( \R^d \), with \( d \geq k \).
Let \( H \) be a hyperplane orthogonal to \( \mathbf{c}_k - \mathbf{c}_{ k + 1 } \), let \( \pi \colon \R^d \to H \) be the orthogonal projection.
Let \( \bar{\mathbf{c}}_k = \pi ( \mathbf{c}_k ) = \pi ( \mathbf{c}_{ k + 1 } ) \) and let \( \bar{\mathbf{c}}_i = \pi ( \mathbf{c}_i ) \) for \( i < k \).
Then \( \bar{\mathbf{c}}_1 , \dots , \bar{\mathbf{c}}_k \) are distinct and in general position in \( H \).
\end{lemma}

\begin{proof}
If \( \bar{\mathbf{c}}_i = \bar{\mathbf{c}}_k \) for some \( i < k \), then \( \mathbf{c}_i , \mathbf{c}_k , \mathbf{c}_{ k + 1 } \) must be collinear, and if \( \bar{\mathbf{c}}_i = \bar{\mathbf{c}}_j \) for some \( i < j < k \), then \( \mathbf{c}_i - \mathbf{c}_j \) and \( \mathbf{c}_k - \mathbf{c}_{ k + 1 } \) must be parallel, so the four points \( \mathbf{c}_i , \mathbf{c}_j , \mathbf{c}_k , \mathbf{c}_{ k + 1 } \) are coplanar.
In either case this contradicts the assumption that \( \mathbf{c}_1 , \dots , \mathbf{c}_{k + 1} \) are in general position.
Therefore the points \( \bar{\mathbf{c}}_1 , \dots , \bar{\mathbf{c}}_k \) are distinct.

In order to conclude the proof it is enough to show that the vectors \( \bar{\mathbf{c}}_1 - \bar{\mathbf{c}}_{k} \), \( \bar{\mathbf{c}}_{2} - \bar{\mathbf{c}}_{k} \), \ldots , \( \bar{\mathbf{c}}_{ k - 1 } - \bar{\mathbf{c}}_{k} \) of \( \R^d \) are linearly independent.
Towards a contradiction, suppose there are \( ( r_1 , \dots , r_{ k - 1 } ) \neq ( 0 , \dots , 0 ) \) such that 
\[ 
\mathbf{0} = r_1 ( \bar{\mathbf{c}}_1 - \bar{\mathbf{c}}_{k} ) + r_2 ( \bar{\mathbf{c}}_2 - \bar{\mathbf{c}}_{k} ) + \dots + r_{ k - 1 }( \bar{\mathbf{c}}_{ k - 1 } - \bar{\mathbf{c}}_{k} ) .
\]
As \( \bar{\mathbf{c}}_i \) is the projection of \( \mathbf{c}_i \) along the vector \( \mathbf{c}_{ k + 1 } - \mathbf{c}_k \), there are \( s_1 , \dots , s_k \in \R \) such that \( \bar{\mathbf{c}}_i = \mathbf{c}_i + s_i ( \mathbf{c}_{ k + 1 } - \mathbf{c}_k ) \).
Substituting in the previous formula we obtain
\begin{multline*}
\mathbf{0} = r_1 ( \mathbf{c}_1 - \mathbf{c}_k ) + \dots + r_{ k - 1 } ( \mathbf{c}_{ k - 1 } - \mathbf{c}_k ) 
\\
{}+ [ r_1 ( s_1 - s_k ) + \dots + r_{ k - 1 } ( s_{k - 1 } - s_k ) ] ( \mathbf{c}_{ k + 1} - \mathbf{c}_k ) ,
\end{multline*}
and since \( \mathbf{c}_1 , \dots , \mathbf{c}_{k + 1} \) are in general position in \( \R^d \), then \( r_1 = \dots = r_{ k - 1} = 0 \) and \( [ r_1 ( s_1 - s_k ) + \dots + r_{ k - 1 } ( s_{k - 1 } - s_k ) ] = 0 \), against our choice of the \( r_i \)s.
\end{proof}

\begin{example}\label{xmp:well-placed}
Suppose \( N \geq d \) and \( \mathbf{c}_1 , \dots , \mathbf{c}_N \) are distinct points of \( \R^d \) such that any \( d \) of them span a subspace of dimension \( d - 1 \).
If \( \mathcal{H}_i \) is the collection of all spheres of \( \R^d \) with center \( \mathbf{c}_i \), then \( ( \mathcal{H}_i )_{ i = 1 }^N \) is locally finite, and of mesh \( d \).
Here is why.

It is enough to consider the case \( N = d \).
If \( d = 1 \) this is trivial since a sphere in \( \R \) is just a pair of points that are symmetric with respect to the center, and if \( d = 2 \) it follows immediately from the fact that the intersection of two circles with distinct centers has size at most \( 2 \).
If \( d = 3 \) this follows from the fact that the intersection of three spheres in \( \R^3 \) with non-collinear centers has size at most \( 2 \).
The case for \( d > 3 \) follows from Theorem~\ref{th:intersectionspheres} below.
\end{example}

\begin{lemma}\label{lem:intersectionspherewithhyperplane}
Suppose \( H \) is an hyperplane of an affine space \( E \) of dimension \( d + 1 \).
Let \( \pi \colon E \to H \) be the orthogonal projection, and let \( h \) be the distance between a point \( \mathbf{c} \in E \) and \( H \).
Then
\[
 \FORALL{R \in \R}\EXISTS{r \in \R}\left ( \mathbb{S} ( E ; \mathbf{c} , R ) \cap H = \mathbb{S} ( H ; \pi ( \mathbf{c} ) , r ) \right )
\]
where
\begin{enumerate}[label={\upshape (\roman*)}]
\item\label{lem:intersectionspherewithhyperplane-i}
\( R > h \IFF R > r > 0 \),
\item\label{lem:intersectionspherewithhyperplane-ii}
\( R = h \IFF r = 0 \), and in this case \( \mathbb{S}( H ; \pi ( \mathbf{c} ) , r ) \) is the singleton \( \setLR{ \pi ( \mathbf{c} ) } \),
\item\label{lem:intersectionspherewithhyperplane-iii}
\( R < h \IFF r < 0 \), that is \( \mathbb{S} ( H ; \pi ( \mathbf{c} ) , r ) = \emptyset \). 
\end{enumerate}

Conversely, if \( \bar{\mathbf{c}} \in H \) and \( \ell \) is the line orthogonal to \( H \) passing through \(\bar{ \mathbf{c}} \), then 
\[
\FORALL{ r > 0 } \FORALL{ \mathbf{c} \in \ell } \EXISTS{R > 0 } \left ( \mathbb{S} ( E ; \mathbf{c} , R ) \cap H = \mathbb{S} ( H , \bar{\mathbf{c}} , r ) \right ) .
\]
\end{lemma}

\begin{proof}
Without loss of generality we may assume that \( E = \R^{ d + 1 } \), that \( H = \R^d \times \set{0} \), \( \mathbf{c} = ( 0 , \dots , 0 , h ) \), and \( \pi ( \mathbf{c} ) = \mathbf{0} \).
It is clear that \( R < h \) if and only if \( \mathbb{S} ( \mathbf{c} , R ) \cap H = \emptyset \), and that \( R = h \) if and only if \( \mathbb{S} ( \mathbf{c} , R ) \cap H = \set{\pi ( \mathbf{c} ) } \), so \ref{lem:intersectionspherewithhyperplane-ii} and \ref{lem:intersectionspherewithhyperplane-iii} are true.
Suppose \( R > h \).
Let \( r = \sqrt{ R^2 - h^2 } \) so that \( r^2 + h^2 = R^2 \).
Then 
\[
\begin{split}
 \mathbf{x} \in \mathbb{S} ( E ; \mathbf{c} , R ) \cap H & \IFF \mathbf{x} = ( x_1 , \dots , x_d , 0 ) \wedge ( \textstyle \sum_{i =1}^d x_i^2 ) + h^2 = R^2
\\
 & \IFF x_1^2 + \dots + x_d ^2 = r^2 
 \\
 & \IFF \| \mathbf{x} \| = r 
\\
 & \IFF \mathbf{x} \in \mathbb{S} ( H ; \pi ( \mathbf{c} ) , r ) . 
\end{split}
\]

For the second part, fix \( \mathbf{c} \in \ell \) and let \( R = \sqrt{ r^2 + h^2 } \) where \( h = \| \mathbf{c} - \bar{ \mathbf{c}} \| \) is the distance from \( \mathbf{c} \) to \( H \), and proceed as before.
\end{proof}

\begin{lemma}\label{lem:intersectionoftwospheres}
Suppose \( \mathbf{c}_1 , \mathbf{c}_2 \) are distinct points of an affine space \( E \) of dimension \( d \geq 2 \), and let \( r_1 , r_2 > 0 \) be such that \( \card{ r_1 - r_2 } \leq \| \mathbf{c}_1 - \mathbf{c}_2 \| \leq r_1 + r_2 \).
Then there exist and are unique \( r , \mathbf{c} , H \) such that 
\begin{itemize}
\item
\( r \) is a real number \( \geq 0 \),
\item
\( \mathbf{c} \) is a point on the open segment \( ( \mathbf{c}_1 ; \mathbf{c}_2 ) \), 
\item
\( H \) is the hyperplane passing through \( \mathbf{c} \) and orthogonal to \( \mathbf{c}_1 - \mathbf{c}_2 \),
\end{itemize}
and
\[ 
\mathbb{S} ( E ; \mathbf{c}_1 , r_1 ) \cap \mathbb{S} ( E ; \mathbf{c}_2 , r_2 ) = \mathbb{S} ( H ; \mathbf{c} , r ) .
\]
\end{lemma}

\begin{proof}
We may assume that \( E = \R^d \).
The hypothesis \( \card{ r_1 - r_2 } \leq \| \mathbf{c}_1 - \mathbf{c}_2 \| \) guarantees that neither sphere properly contains inside the other sphere, while the hypothesis \( \| \mathbf{c}_1 - \mathbf{c}_2 \| \leq r_1 + r_2 \) ensures that if neither sphere contains the other, then they are close enough to intersect.

Let \( \mathbf{x} \) be a point of \( \mathbb{S} ( \mathbf{c}_1 , r_1 ) \cap \mathbb{S} ( \mathbf{c}_2 , r_2 ) \), and let \( P \) be the plane passing through \( \mathbf{c}_0 , \mathbf{c}_1 , \mathbf{x} \).
(If \( d = 2 \) then \( P \) coincides with \( \R^2 \).)
Consider the triangle with vertices \( \mathbf{c}_0 , \mathbf{c}_1 , \mathbf{x} \).

If \( r_1 + r_2 = \| \mathbf{c}_1 - \mathbf{c}_2 \| \) then the two spheres are tangent in the point \( \mathbf{c} \coloneqq \mathbf{x} \) which belongs to the open segment \( ( \mathbf{c}_1 ; \mathbf{c}_2 ) \), so the triangle is degenerate and letting \( r = 0 \) we have that \( \mathbb{S} ( H ; \mathbf{c} , r ) = \set{\mathbf{c} } \).

So we may assume that \( r_1 + r_2 > \| \mathbf{c}_1 - \mathbf{c}_2 \| \) and that the triangle 
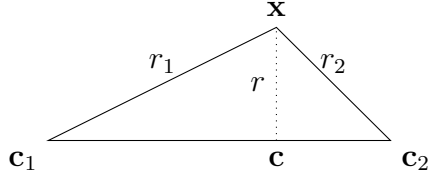
\[
\begin{tikzpicture}[scale=1.5]
\node (c1) at (0,0) [anchor=north east] {\( \mathbf{c}_1 \)};
\node (c2) at (3,0) [anchor= north west] {\( \mathbf{c}_2 \)};
\node (x) at (2,1) [anchor= south ] {\( \mathbf{x} \)};
\node (c) at (2,0) [anchor= north ] {\( \mathbf{c} \)};
\draw[-] (0,0) -- (3,0)--(2,1)--cycle;
\draw [-, dotted] (2,0)--(2,1);
\node at (1, 0.5) [anchor=south]{\( r_1 \)};
\node at (2.5, 0.5) [anchor=south]{\( r_2 \)};
\node at (2, 0.5) [anchor=east]{\( r \)};
\end{tikzpicture}
\]
is non-degenerate.
Let \( \mathbf{c} \) is the point on the open segment \( ( \mathbf{c}_1 ; \mathbf{c}_2 ) \) whose distance from \( \mathbf{c}_i \) is \( \sqrt{r_i^2 - r^2} \), and let \( H \) be the hyperplane orthogonal to \( \mathbf{c}_1 - \mathbf{c}_2 \) passing through \( \mathbf{c} \).
Observe that \( r \), the point \( \mathbf{c} \), and the hyperplane \( H \), depend only on the points \( \mathbf{c}_1 \) and \( \mathbf{c}_2 \), and not on the point \( \mathbf{x} \) or the plane \( P \).
Moreover \( r = \| \mathbf{x} - \mathbf{c} \| \).
Therefore for all \( \mathbf{x} \in \R^d \),
\[
\begin{split}
\mathbf{x} \in \mathbb{S} ( \mathbf{c}_1 , r_1 ) \cap \mathbb{S} ( \mathbf{c}_2 , r_2 ) & \IFF \| \mathbf{x} - \mathbf{c}_1 \| = r_1 \wedge \| \mathbf{x} - \mathbf{c}_2 \| = r_2
\\
 & \IFF \mathbf{x} \in H \wedge \| \mathbf{c} - \mathbf{x} \| = r .
\end{split}
\]
The argument above shows that \( r , \mathbf{c} , H \) are unique.
\end{proof}

\begin{proposition}\label{prop:intersectionofspheres}
Let \( E \) be an affine space of dimension \( d \geq 2 \), let \( \mathbf{c}_1 , \dots , \mathbf{c}_k \) be in general position in \( E \) with \( k \leq d \), and let \( K \) be the affine span of the \( \mathbf{c}_i \)s.
Let \( S_i \) be a sphere centered in \( \mathbf{c}_i \) for \( i \leq k \).
\begin{enumerate}[label={\upshape (\alph*)}]
\item\label{prop:intersectionofspheres-a}
There are \( r \in \R \) and an affine subspace \( H \) of \( E \) of dimension \( d - ( k - 1 ) \), such that \( H \cap K \) is a singleton \( \setLR{ \mathbf{c}} \) and \( H \) and \( K \) are orthogonal, so that \( S_1 \cap \dots \cap S_k = \mathbb{S} ( H ; \mathbf{c} , r ) \).
\item\label{prop:intersectionofspheres-b}
If \( \mathbf{c}_{ k + 1 } \in E \setminus \set{ \mathbf{c}_1 , \dots , \mathbf{c}_k } \) and \( \set{ \mathbf{c}_1 , \dots , \mathbf{c}_k ,\mathbf{c}_{ k + 1 } } \) is not in general position in \( E \), then there is \( R \in \R \) such that
\[
 S_1 \cap \dots \cap S_k \subseteq \mathbb{S} ( \mathbf{c}_{ k + 1 } , R ) .
\] 
\end{enumerate}
\end{proposition}

\begin{proof}
We may assume that \( E = \R^d \).

If \( d = 2 \) then the result is clear.
Part~\ref{prop:intersectionofspheres-a} amounts to say that: given two circles \( C_1 , C_2 \) in the plane with distinct centers \( \mathbf{c}_1 \), \( \mathbf{c}_2 \), either \( I \coloneqq C_1 \cap C_2 \) is empty, or else there is a point \( \mathbf{c} \) on the open segment \( ( \mathbf{c}_1 ; \mathbf{c}_2 ) \) such that \( I = \set{\mathbf{c}} \), or else \( I \) is the set of two points on line orthogonal to \( ( \mathbf{c}_1 ; \mathbf{c}_2 ) \) passing through \( \mathbf{c} \), and symmetric with respect to \( \mathbf{c} \).
Part~\ref{prop:intersectionofspheres-b} says that if three points are collinear, then given any two circles centered in the first two points there is a (possibly degenerate) circle centered in the third point that passes through the intersection of the first two circles.

Therefore we may assume that \( d \geq 3 \).

\smallskip

\ref{prop:intersectionofspheres-a} 
We proceed by induction on \( k \).
If \( k = 1 \) the result is trivial, so we may assume the result holds for some \( k \) towards proving it for \( k + 1 < d \).
Suppose \( S_1 , \dots , S_{ k + 1 } \) are spheres centered in distinct points \( \mathbf{c}_1 , \dots , \mathbf{c}_{ k + 1} \).
By an isometry we may assume that \( K \), the affine span of the \( \mathbf{c}_i \)s, is contained in \( \R^ { d - 1 } \times \set{0} \).
Let \( \bar{S}_k \coloneqq S_k \cap S_{k + 1} \).
We have three cases:
\begin{enumerate}[label={\upshape (\arabic*)}]
\item\label{en:prop:intersectionofspheres-1}
\( \bar{S}_k = \emptyset \),
\item\label{en:prop:intersectionofspheres-2}
\( \bar{S}_k \) is a singleton,
\item\label{en:prop:intersectionofspheres-3}
\( \bar{S}_k \) is a non-degenerate sphere in some hyperplane \( H \) of \( \R^d \) (Lemma~\ref{lem:intersectionoftwospheres}).
\end{enumerate} 
If either~\ref{en:prop:intersectionofspheres-1} or~\ref{en:prop:intersectionofspheres-2} holds, then \( S_1 \cap \cdots \cap S_{k + 1} \) is empty or a singleton, and the result follows trivially.
So we may assume case~\ref{en:prop:intersectionofspheres-3}.
The hyperplane \( H \) is orthogonal to the open segment \( ( \mathbf{c}_k ; \mathbf{c}_{ k + 1} ) \), and \( \bar{\mathbf{c}}_k \), the center of \( \bar{S}_k \), belongs to this segment.
Let \( \pi \colon \R^d \to H \) be the orthogonal projection.
Then \( \pi ( \mathbf{c}_k ) = \pi ( \mathbf{c}_{ k + 1 } ) = \bar{\mathbf{c}}_k \) and for \( i < k \) let \( \bar{\mathbf{c}}_i \coloneqq \pi ( \mathbf{c}_i ) \).
By Lemma~\ref{lem:intersectionspherewithhyperplane} for every \( i < k \) the set \( \bar{S}_i \coloneqq S_i \cap H \) is a sphere of \( H \) with center \( \bar{\mathbf{c}}_i \).
As \( S_k \cap S_{ k + 1} \subseteq H \) 
\[
S_1 \cap \dots \cap S_k \cap S_{ k + 1 } = \bar{S}_1 \cap \dots \cap \bar{S}_k .
\]
By Lemma~\ref{lem:projectiogeneralposition} \( \bar{\mathbf{c}}_1 , \dots , \bar{\mathbf{c}}_k \) are in general position in \( H \cong \R^{ d - 1} \), and since \( d - 1 \geq 2 \), by inductive assumption we can conclude that \( \bar{S}_1 \cap \dots \cap \bar{S}_k = \mathbb{S} ( H' ; \mathbf{c} , r ) \) where \( H' \) a subspace of \( H \) (and hence of \( \R^d \)) of dimension \( ( d - 1 ) - ( k - 1 ) = d - k \), such that \( \setLR{\mathbf{c} } = H' \cap K \), and \( H' \) is orthogonal to \( K \).

\smallskip

\ref{prop:intersectionofspheres-b} 
By assumption \( \mathbf{c}_{ k + 1 } \) belongs to \( K \), and \( k - 1 = \dim ( K ) \).
By part~\ref{prop:intersectionofspheres-a} there is a subspace \( H \) orthogonal to \( K \), of dimension \( d - ( k - 1 ) \), and \( r \in \R \) such that \( S_1 \cap \dots \cap S_k = \mathbb{S} ( H ; \mathbf{c} , r ) \) where \( \mathbf{c} \) is the unique element of \( H \cap K \).
It is enough to show that 
\[
\EXISTS{ R \in \R } \bigl ( \mathbb{S} ( H ; \mathbf{c} , r ) \subseteq S_{ k + 1 } \coloneqq \mathbb{S} ( E ; \mathbf{c}_{ k + 1 } , R ) \bigr ) .
\]
This is immediate if \( r \leq 0 \), so we may assume that \( r > 0 \), that is to say \( \mathbb{S} ( H ; \mathbf{c} , r ) \) is non-degenerate.
If \( \mathbf{c} = \mathbf{c}_{ k + 1 } \) then taking \( R = r \) the result follows at once, so we may assume that \( \mathbf{c} \neq \mathbf{c}_{ k + 1 } \).
Let \( R = \sqrt{ \| \mathbf{c}_{ k + 1 } - \mathbf{c} \|^2 + r^2 } \).
Given \( \mathbf{x} \in \mathbb{S} ( H ; \mathbf{c} , r ) \), we must show that
\[
\| \mathbf{x} - \mathbf{c}_{ k + 1 } \|^2 = \| \mathbf{x} - \mathbf{c} \|^2 + \| \mathbf{c}_{ k + 1 } - \mathbf{c} \|^2
\] 
so that \( \mathbf{x} \in \mathbb{S} ( E ; \mathbf{c}_{ k + 1 } , R ) \).
But this is immediate, since \( H , K \) are orthogonal and hence \( \mathbf{c}_{ k + 1 } - \mathbf{c} \) and \( \mathbf{x} - \mathbf{c} \) are orthogonal.
\end{proof}

\begin{theorem}\label{th:intersectionspheres}
Let \( E \) be an affine space of dimension \( d \geq 2 \) and let \( \mathbf{c}_1 , \dots , \mathbf{c}_d \) be distinct points in general position in \( E \).
\begin{enumerate}[label={\upshape (\alph*)}]
\item\label{th:intersectionspheres-a}
For all spheres \( S_i \) in \( E \) centered in \( \mathbf{c}_i \) with \( 1 \leq i \leq d \), the set \( S_1 \cap \cdots \cap S_d \) is finite.
\item\label{th:intersectionspheres-b}
For all \( k < d \) and \( 1 \leq i \leq k \) there are spheres \( S_i \) in \( E \) centered in \( \mathbf{c}_i \) such that \( S_1 \cap \cdots \cap S_k \) is a non-degenerate sphere in a subspace of dimension \( d - ( k - 1 ) \geq 2 \).
\end{enumerate}
\end{theorem}

\begin{proof}
Recall that a non-degenerate sphere in space of dimension \( n \) has cardinality \( 2^{\aleph_0} \) or \( 2 \) depending whether \( n > 1 \) or \( n = 1 \).
Part~\ref{th:intersectionspheres-a} then follows at once form Proposition~\ref{prop:intersectionofspheres}.

\ref{th:intersectionspheres-b}
We prove by induction on \( k \) a stronger statement: for any \( r_k > 0 \) and for any affine space \( E \) of dimension \( d > k \), there exist \( r_1 , \dots , r_{k- 1} > 0 \) such that 
\[
\mathbb{S} ( E ; \mathbf{c}_1 , r_1 ) \cap \dots \cap \mathbb{S} ( E ; \mathbf{c}_k , r_k ) = \mathbb{S} ( H ; \mathbf{c} , r ) ,
\] 
for some \( r > 0 \), \( \mathbf{c} \in H \) and \( H \) a subspace of \( E \) of dimension \( d - ( k - 1 ) \).

The base case \( k = 1 \) is immediate, so we may assume that the result holds for some \( k \) towards proving it for \( k + 1 < d \).
Fix \( r_{ k + 1 } > 0 \) and pick \( r_k > 0 \) such that \( \card{ r_k - r_{ k + 1 } } < \| \mathbf{c}_k - \mathbf{c}_{ k + 1 } \| < r_k + r_{ k + 1 } \).
By Lemma~\ref{lem:intersectionoftwospheres} \( \mathbb{S} ( E ; \mathbf{c}_{ k + 1 } , r _{ k + 1 } ) \cap \mathbb{S} ( E ; \mathbf{c}_k , r _k ) \) is a non-degenerate sphere \( \mathbb{S} ( \bar{E} ; \bar{\mathbf{c}}_{ k } , \bar{r} _{ k } ) \) in some hyperplane \( \bar{E} \) of \( E \).
Let \( \pi \colon E \to \bar{E} \) be the orthogonal projection. 
By Lemma~\ref{lem:projectiogeneralposition} the points
\begin{align*}
 \bar{\mathbf{c}}_1 & = \pi ( \mathbf{c}_1 ) , & \bar{\mathbf{c}}_2 & = \pi ( \mathbf{c}_2 ) , & &\dots & \bar{\mathbf{c}}_{k - 1} & = \pi ( \mathbf{c}_{ k - 1} ) , & \bar{\mathbf{c}}_k & = \pi ( \mathbf{c}_k ) = \pi ( \mathbf{c}_{ k + 1 } )
\end{align*}
are distinct and in general position in the affine space \( \bar{E} \) of dimension \( d - 1 > k \).
Since \( \bar{r} _{ k } > 0 \), by inductive assumption there exist \( \bar{r}_1 , \dots , \bar{r}_{ k - 1} > 0 \) such that 
\[
 S \coloneqq \mathbb{S} ( \bar{E} , \bar{\mathbf{c}}_1 , \bar{r}_1 ) \cap \dots \cap \mathbb{S} ( \bar{E} , \bar{\mathbf{c}}_k , \bar{r}_k ) 
\]
is a non-degenerate sphere in a subspace \( H \) of dimension \( ( d - 1 ) - ( k - 1 ) = d - k \).
Letting \( r_i = \sqrt{ \bar{r}_i^2 + h_i^2 } \) where \( h_i \) is the distance between \( \mathbf{c}_i \) and the hyperplane \( \bar{E} \), we have that
\[
 S = \mathbb{S} (E , \mathbf{c}_1 , r_1 ) \cap \dots \cap \mathbb{S} ( E , \mathbf{c}_k , r_k ) \cap \mathbb{S} ( E , \mathbf{c}_{k + 1} , r_{k + 1} )
\]
as required.
\end{proof}

\begin{corollary}\label{cor:intersectionspheres-1}
Suppose \( E \) is an affine space of dimension \( d \geq 2 \) and \( \mathbf{c}_1 , \dots , \mathbf{c}_{ d + 1} \) are distinct points that are not in general position in \( E \).
Then there are spheres \( S_i \) centered in \( \mathbf{c}_i \) such that \( S_1 \cap \dots \cap S_{ d + 1} \) is infinite.
\end{corollary}

\begin{proof}
Let \( k \) be the dimension of the affine span \( K \) of \( \setLR{ \mathbf{c}_1 , \dots , \mathbf{c}_{ d + 1 } } \).
By assumption \( k < d \), and without loss of generality we may assume that \( \mathbf{c}_1 , \dots , \mathbf{c}_{k + 1} \) are in general position in \( K \).
By part~\ref{th:intersectionspheres-b} of Theorem~\ref{th:intersectionspheres} there are spheres \( S_1 , \dots , S_{ k + 1 } \) centered in \( \mathbf{c}_1 , \dots , \mathbf{c}_{ k + 1 } \) whose intersection is a non-degenerate sphere in a space of dimension \( \geq 2 \).
By repeated applications of part~\ref{prop:intersectionofspheres-b} of Proposition~\ref{prop:intersectionofspheres}, there are spheres \( S_{ k + 2 }, \dots , S_{ d + 1 } \) centered in \( \mathbf{c}_{ k + 2 } , \dots , \mathbf{c}_{ d + 1 } \) such that 
\[
 S_1 \cap \dots \cap S_{ k + 1 } \subseteq S_i \quad \text{for } k + 2 \leq i \leq d + 1
\]
and hence \( S_1 \cap \dots \cap S_{ k + 1 } = S_1 \cap \dots \cap S_{ d + 1 } \) is infinite.
\end{proof}

Observe that by~\eqref{eq:properlyplaced}, both ``\( S \) is in general position'' and ``\( S \) is well-placed'' imply the assumption of the next corollary.

\begin{corollary}\label{cor:intersectionspheres-2}
Suppose \( S \) is a set of at least \( d \) points of an affine space \( E \) of dimension \( d \geq 2 \).
Assume that any \( d \)-many points of \( S \) span a subspace of dimension \( d - 1 \).
Then the intersection of \( d \)-many spheres centered in these points is a finite set.
\end{corollary}

\subsection{Sprays}\label{subsec:sprays}
J.~Schmerl in~\cite{Schmerl:2003uq} defined a spray to be a set \( X \subseteq \R^2 \) such that \( C \cap X \) is finite, for all circles \( C \) centered in some given point.
Thus a spray \( X \) is a very sparse subset of the plane, and one may investigate what happens if one relaxes the finiteness condition on \( C \cap X \) with, say, being countable.

\begin{definition}\label{def:spray}
Let \( E \) be an affine subspace of \( \R^d \), with \( d \geq 2 \) be a natural number.
An \markdef{\( \aleph_ \delta \)-spray in \( E \) with center \( \mathbf{c} \)} is a set \( X \subseteq E \) such that \( \card{ \mathbb{S} ( E ; \mathbf{c} , r ) \cap X } < \aleph_ \delta \) for every \( r \in \R \).
When \( \delta = 0 \) it is called a \markdef{spray}, when \( \delta = 1 \) we speak of \markdef{\( \sigma \)-spray}.
\end{definition}

Note that a center need not to belong to the \( \aleph_ \delta \)-spray, and that a \( \aleph_ \delta \)-spray may have more than one center.

By Example~\ref{xmp:well-placed}, if \( \mathbf{c}_1 , \dots , \mathbf{c}_{ ( n + 1 ) ( d - 1 ) + 1 } \) are either in general position, or well-placed in \( \R^d \), then 
\begin{itemize}
\item
\( 2^{\aleph_0} \leq \aleph_n \) implies that there are \( X_1 , \dots , X_{ ( n + 1 ) ( d - 1 ) + 1 } \) covering \( \R^d \) such that each \( X_i \) is a spray centered in \( \mathbf{c}_i \);
\item
\( 2^{\aleph_0} \leq \aleph_{ n + 1 } \) implies that there are \( X_1 , \dots , X_{ ( n + 1 ) ( d - 1 ) + 1 } \) covering \( \R^d \) such that each \( X_i \) is a \( \sigma \)-spray.
\end{itemize}

In particular if the points are well-placed, that is if we assume that the \( \mathbf{c}_i \)s lie on a hyperplane, then:

\begin{theorem}\label{th:boundoncontinuum=>R^dcoveredwithkappasprays}
Let \( d \geq 2 \), \( n \geq 0\), and \( \delta = 0 , 1 \).
Assume that \( \mathbf{c}_1 , \dots , \mathbf{c}_{ ( n + 1 ) (d - 1 ) + 1 } \) are well-placed points in \( \R^d \).
If \( 2^{\aleph_0 } \leq \aleph_{ \delta + n } \) then there are \( X_1 , \dots , X_{ ( n + 1 ) ( d - 1 ) + 1 } \) covering \( \R^d \) such that each \( X_i \) is an \( \aleph_ \delta \)-spray centered in \( \mathbf{c}_i \).
\end{theorem}

By~\eqref{eq:properlyplaced} and Corollary~\ref{cor:intersectionspheres-2}, the thesis of Theorem~\ref{th:boundoncontinuum=>R^dcoveredwithkappasprays} holds if \( \mathbf{c}_1 , \dots , \mathbf{c}_{ ( n + 1 ) (d - 1 ) + 1 }\) are in general position in \( \R^d \).
The reason for focusing on well-placed points is that under this assumption the implication in Theorem~\ref{th:boundoncontinuum=>R^dcoveredwithkappasprays} can be reversed, and this is the main goal of this paper.
In Section~\ref{sec:Transforming-sprays-into-linear-objects} we argue that any covering of \( \R^d \) with sprays (\( \sigma \)-sprays) with well-placed centers \( \mathbf{c}_i \)s can be transformed into a covering (of an open subset) of \( \R^d \) with sets whose intersection with the hyperplanes orthogonal to \( \mathbf{u}_i \) is finite (countable) for all \( i \), where the vector \( \mathbf{u}_i \) is obtained from the point \( \mathbf{c}_i \).
Appealing to the results from Section~\ref{sec:Hyperplane-sections} the equivalence will be established.

The case of points in general position, but not well-placed, i.e. not belonging to an hyperplane, is more problematic.
If \( d = 2 \) and three points are not collinear (i.e. are in general position, but not well-placed in \( \R^2 \)) then by~\cite[Theorem 1]{Schmerl:2010nr}, the plane is the union of three sprays centered in these points, irrespective of the size of the continuum.
The analogous statement for \( d > 2 \) is open, and if true it would show that the assumption that the points are well-placed cannot be dropped.
\begin{conjecture}
For all \( d \geq 3 \), for all \( \mathbf{c}_1 , \dots , \mathbf{c}_{d + 1 } \) in general position in \( \R^d \), there are \( X_1 , \dots , X_{ d + 1 } \) covering \( \R^d \), with \( X_i \) a spray centered in \( \mathbf{c}_i \) for \( 1 \leq i \leq d + 1 \).
\end{conjecture}

\subsection{A few elementary results about sprays} 
Observe the following very basic, but useful facts.
Let \( \kappa \) be an infinite cardinal.

\begin{description}
\item[Glueing]
If \( X \) and \( Y \) are \( \aleph_ \delta \)-sprays in an affine space \( E \) with the same center \( \mathbf{c} \), then \( X \cup Y \) is a \( \aleph_ \delta \)-spray in \( E \) with center \( \mathbf{c} \).
\item[Projection]
Suppose that \( X_1 , \dots , X_n \) are \( \aleph_ \delta \)-sprays in an affine space \( E \) of dimension \( d + 1 \), and \( X_1 \cup \dots \cup X_n = E \) with centers \( \mathbf{c}_1 , \dots , \mathbf{c}_n \).
Suppose \( H \) is an hyperplane of \( E \), and let \( \pi \colon E \to H \) be the orthogonal projection.
Then \( X_1 \cap H , \dots , X_n \cap H \) are \( \aleph_ \delta \)-sprays in \( H \) with centers \( \pi ( \mathbf{c}_1 ) , \dots , \pi ( \mathbf{c}_n ) \in H \), that cover \( H \) .
\end{description}

Projection follows from Lemma~\ref{lem:intersectionspherewithhyperplane}.
Observe that the points \( \pi ( \mathbf{c}_1 ) , \dots , \pi ( \mathbf{c}_n ) \) need not be distinct, even if the \( \mathbf{c}_1 , \dots , \mathbf{c}_n \) are distinct.
This, together with glueing, allows to transform \( n \)-many \( \aleph_ \delta \)-sprays covering \( \R^{ d + 1 } \) into \( n' \)-many \( \aleph_ \delta \)-sprays covering \( \R^d \), with \( n' < n \).

\begin{proposition}\label{prop:R2notunionof2sprays}
Suppose \( \R^2 = X_1 \cup X_2 \) where \( X_1 \) is an \( \aleph_ \delta \)-spray and \( X_2 \) is an \( \aleph_ { \delta + 1 } \)-spray.
Then \( 2^{\aleph_0} \leq \aleph_ \delta \).
\end{proposition}

\begin{proof}
For \( i = 1 , 2 \), let \( \mathbf{c}_i \) be the center of \( X_i \), and for \( r > 0 \) let
\[
 C_i ( r ) \coloneqq \setof{ \mathbf{p} \in \R^2 }{ \| \mathbf{p} - \mathbf{c}_i \| = r }
\] 
be the circle centered in \( \mathbf{c}_i \) of radius \( r \).
If \( \mathbf{c}_1 = \mathbf{c}_2 \) then by glueing \( X_1 \cup X_2 \) would be an \( \aleph_ { \delta + 1 } \)-spray in \( \R^2 \), and every circle centered in \( \mathbf{c}_1 \) is contained in \( X_1 \cup X_2 = \R^2 \) is of cardinality \( \leq \aleph_ \delta \).
As any circle is in bijection with \( \R \), the result follows.
Therefore we may assume that \( \mathbf{c}_1 \neq \mathbf{c}_2 \).
Applying an isometry if needed, we may assume that \( \mathbf{c}_1 = ( 0 , 0 ) \) and \( \mathbf{c}_2 = ( a , 0 ) \) for some \( a > 0 \). 

Towards a contradiction, suppose \( 2^{\aleph_0} > \aleph_ \delta \).
Fix distinct reals \( r_ \alpha \) in the interval \( ( a / 2 ; a ) \), for \( \alpha < \aleph_ \delta \).
By assumption \( X_ 2 ( \alpha ) \coloneqq X_2 \cap C_2 ( r_ \alpha ) \) has size \( \leq \aleph_ \delta \), so the set
\[
\textstyle \setof{ r \in ( a / 2 ; a ) }{ C_1 ( r ) \cap ( \bigcup_{ \alpha \in \aleph_ \delta } X_2 ( \alpha ) ) \neq \emptyset }
\]
has size \( \leq \aleph_ \delta \).
As \( \card{\R} > \aleph_ \delta \) we may pick \( r \in ( a / 2 ; a ) \) outside of this set.
For each \( \alpha \in \aleph_ \delta \) the set \( C_1 ( r ) \cap C_2 ( r_ \alpha ) \) has size \( 2 \), and its points belong to \( X_1 \).
As the \( C_2 ( r_ \alpha ) \)s are disjoint, it follows that \( C_1 ( r ) \cap X_1 \) has size \( \geq \aleph_ \delta \), contradicting that \( X_1 \) is an \( \aleph_ \delta \)-spray.
\end{proof}

When \( \delta = 0 \) we obtain at once:

\begin{corollary}\label{cor:R2notunionof2sprays}
\( \R^2 \) is not the union of a spray and a \( \sigma \)-spray.
In particular, \( \R^2 \) is not the union of two sprays.
\end{corollary}

\begin{corollary}\label{cor:R2unionof2sigmasprays}
The following are equivalent:
\begin{enumerate}[label={\upshape (\alph*)}]
\item
\( \CH \)
\item
\( \R^2 \) is the union of two \( \sigma \)-sprays with prescribed, distinct centers.
\end{enumerate}
\end{corollary}

\begin{proposition}\label{prop:Rdnotunionofdsprays}
Let \( n \geq d \geq 3 \), and suppose that the points \( \mathbf{c}_1 , \dots , \mathbf{c}_n \) belong to a hyperplane \( H \) of \( \R^d \).
Suppose there is \( L \subseteq H \) an affine subspace of dimension \( d - 2 \) such that \( \set{ \pi ( \mathbf{c}_1 ) , \dots , \pi ( \mathbf{c}_n ) } \) has size \( \leq d - 1 \), where \( \pi \colon H \to L \) is the orthogonal projection.
Then there are no sprays \( X_1 , \dots , X_n \) that cover \( \R^d \) with \( X_i \) centered in \( \mathbf{c}_i \), for \( i \leq n \).
\end{proposition}

\begin{proof}
Towards a contradiction, let \( d \geq 3 \) be least such that the statement fails, and suppose \( X_i \) is a spray centered in \( \mathbf{c}_i \) such that \( \R^d = X_1 \cup \dots \cup X_n \).
Then \( H \cong \R^{ d - 1 } \) is covered by the sprays \( \pi [ X_1 ] , \dots , \pi [ X_n ] \) centered in \( \set{ \pi ( \mathbf{c}_1 ) , \dots , \pi ( \mathbf{c}_n ) } \).
This means that \( \R^{ d - 1} \) can be covered by \( d - 1 \)-many sprays whose centers lie on the hyperplane \( L \).
This contradicts the minimality of \( d \), if \( d - 1 \geq 3 \).
If \( d = 3 \), then we would have that \( \R^2 \) can be covered by two sprays, against Corollary~\ref{cor:R2notunionof2sprays}.
\end{proof}

In~\cite[p.~1169]{Schmerl:2012aa} it is observed that the next result follows from results of Sikorski.

\begin{theorem}\label{th:Rdnotunionofdsprays}
For \( d \geq 2 \) the space \( \R^d \) is not the union of \( d \)-many sprays.
\end{theorem}

\begin{proof}
The case \( d = 2 \) is Corollary~\ref{cor:R2notunionof2sprays}.
Suppose \( d \geq 3 \) and that \( \mathbf{c}_1 , \dots , \mathbf{c}_d \) are centers of sprays \( X_1 , \dots , X_d \) that cover \( \R^d \).
Let \( H \) be a hyperplane containing these points, and apply Proposition~\ref{prop:Rdnotunionofdsprays} with \( n = d - 1 \): if \( L \) is an affine subspace of \( H \) of dimension \( d - 2 > 0 \) that is orthogonal to the vector \( \mathbf{c}_d - \mathbf{c}_{ d - 1} \), then \( \set{ \pi ( \mathbf{c}_1 ) , \dots , \pi ( \mathbf{c}_d ) } \) has size \( \leq d-1 \) since \( \pi ( \mathbf{c}_d ) = \pi ( \mathbf{c}_{ d - 1} ) \).
But Proposition~\ref{prop:Rdnotunionofdsprays} implies that the \( X_i \)s cannot cover \( \R^d \).
\end{proof}

A similar argument shows:

\begin{theorem}\label{th:Rdnotunionofd-1sigmasprays}
For \( d \geq 2 \) the space \( \R^d \) is not the union of \( ( d - 1 ) \)-many \( \sigma \)-sprays.
\end{theorem}

\begin{theorem}\label{th:sigmaspraysinRd}
For \( d \geq 2 \) the following are equivalent:
\begin{enumerate}[label={\upshape (\alph*)}]
\item\label{th:sigmaspraysinRd-a}
\( \CH \);
\item\label{th:sigmaspraysinRd-b}
for all well-placed \( \mathbf{c}_1 , \dots , \mathbf{c}_d \in \R^d \) there are \( X_1 , \dots , X_d \) covering \( \R^d \) such that each \( X_i \) is a \( \sigma \)-spray with center \( \mathbf{c}_i \);
\item\label{th:sigmaspraysinRd-c}
there are well-placed \( \mathbf{c}_1 , \dots , \mathbf{c}_d \in \R^d \) and \( X_1 , \dots , X_d \) covering \( \R^d \) such that each \( X_i \) is a \( \sigma \)-spray with center \( \mathbf{c}_i \).
\end{enumerate}
\end{theorem}

\begin{proof}
\ref{th:sigmaspraysinRd-a}\( \IMPLIES \)\ref{th:sigmaspraysinRd-b} follows from Theorem~\ref{th:boundoncontinuum=>R^dcoveredwithkappasprays}; \ref{th:sigmaspraysinRd-b}\( \IMPLIES \)\ref{th:sigmaspraysinRd-c} is trivial, while \ref{th:sigmaspraysinRd-c}\( \IMPLIES \)\ref{th:sigmaspraysinRd-a} is established by induction on \( d \geq 2 \).

When \( d = 2 \) the result follows at once from Proposition~\ref{prop:R2notunionof2sprays}.
Suppose the result holds for some \( d \) towards proving it for \( d + 1 \).
Fix \( \sigma \)-sprays \( X_1 , \dots , X_{ d + 1 } \) covering \( \R^{ d + 1 } \) and centered in well-placed \( \mathbf{c}_1 , \dots , \mathbf{c}_{ d + 1 } \).
Let \( H \subseteq \R^d \) be the hyperplane determined by the \( \mathbf{c}_i \)s.
Let \( H' \) be a hyperplane orthogonal to \( \mathbf{c}_{d + 1} - \mathbf{c}_d \), and let \( \pi \colon \R^{ d + 1} \to H ' \) be the orthogonal projection.
By projecting and glueing \( X_1 \cap H \), \( X_1 \cap H \), \ldots , \( ( X_d \cup X_{ d + 1} ) \cap H \) are \( \sigma \)-sprays centered in the points \( \pi ( \mathbf{c}_1 ) \), \( \pi ( \mathbf{c}_2 ) \), \ldots , \( \pi ( \mathbf{c}_d ) = \pi ( \mathbf{c}_{ d + 1 } ) \) which are well-placed, and belong to \( H' \cong \R^d \).
By inductive assumption, \( \CH \) holds.
\end{proof}

\section{Transforming sprays into linear objects}\label{sec:Transforming-sprays-into-linear-objects}
In this section we construct, for every \( d \geq 2 \), a continuous map \( \Phi \) that transforms any spray of \( \R^d \) with center \( \mathbf{c} \) into a set \( A \subseteq \R^d \) such that \( A \cap H \) is finite, for every hyperplane \( H \) orthogonal to some vector \( \mathbf{u} \), and conversely.
(The vector \( \mathbf{u} \) depends only on the point \( \mathbf{c} \).)
This is an extension and an elaboration of the construction used by Schmerl when \( d = 2 \) to prove that if \( \R^2 \) is the union of \( ( n + 2 ) \)-many sprays with collinear centers, then \( 2^{\aleph_0} \leq \aleph_n \)~\cite[Theorem 7]{Schmerl:2010nr}.

Let \( d \geq 2 \), and fix distinct points \( \mathbf{p}_1 , \dots , \mathbf{p}_d \) in \( \R^{ d - 1} \).
For each \( \mathbf{q} \in \R^{ d - 1} \) the vectors \( \mathbf{p}_1 - \mathbf{q} , \dots , \mathbf{p}_d - \mathbf{q} \) are linearly dependent, and hence 
\begin{equation}\label{eq:U}
\mathcal{U} ( \mathbf{q} ) \coloneqq \setof{ ( u_1 , \dots , u_d ) \in \R^d}{ u_1 ( \mathbf{p}_1 - \mathbf{q} ) + \dots + u_d ( \mathbf{p}_d - \mathbf{q} ) = \mathbf{0} } 
\end{equation}
is a vector subspace of \( \R^d \) of dimension \( \geq 1 \).
Clearly this space depends on the \( \mathbf{p}_i \)s, so we should write \( \mathcal{U}_{\mathbf{p}_1, \dots , \mathbf{p}_d } ( \mathbf{q} ) \), but the \( \mathbf{p}_i \)s will be clear from the context.
Observe that if \( \mathbf{e}_i \in \mathcal{U} ( \mathbf{q} ) \) then \( \mathbf{q} = \mathbf{p}_i \), and hence if \( \mathbf{q} \) is distinct from \( \mathbf{p}_1 , \dots , \mathbf{p}_d \) then no \( \mathbf{u} \in \mathcal{U} ( \mathbf{q} ) \setminus \set{ \mathbf{0}} \) is parallel to a basis vector of \( \R^d \).

\begin{theorem}\label{th:Ivan}
Suppose \( \mathbf{p}_1 , \dots , \mathbf{p}_d , \mathbf{q} \in \R^{ d - 1 } \).
For every \( ( u_1 , \dots , u_d ) \in \mathcal{U} ( \mathbf{q} ) \setminus \set{ \mathbf{0} } \) letting \( b \coloneqq - \sum_{ i = 1 }^{ d } u_i \) and \( c \coloneqq - \sum_{ i = 1 }^{ d } u_i ( \| \mathbf{p}_i \|^2 - \| \mathbf{q} \|^2 ) \) we have that
\[
\FORALL{ \mathbf{x} \in \R^{ d - 1} } \bigl ( u_1 \| \mathbf{x} - \mathbf{p}_1 \|^2 + \dots + u_d \| \mathbf{x} - \mathbf{p}_d \|^2 + b \| \mathbf{x} - \mathbf{q} \|^2 + c = 0 \bigr ) .
\]
\end{theorem}

\begin{proof}
For notational ease let \( F_i ( \mathbf{x} ) = \| \mathbf{x} - \mathbf{p}_i \|^2 \) and \( F ( \mathbf{x} ) = \| \mathbf{x} - \mathbf{q} \|^2 \).
For \( \mathbf{x} \in \R^{ d - 1} \) and \( 1 \leq i \leq d \)
\[
\begin{split}
F_i ( \mathbf{x} ) - F ( \mathbf{x} ) & = \| \mathbf{x} - \mathbf{p}_i \|^2 - \| \mathbf{x} - \mathbf{q} \|^2
\\
 & = \| \mathbf{x} \|^2 + \| \mathbf{p}_i \|^2 - 2 \mathbf{x} \boldsymbol{\cdot} \mathbf{p}_i - \| \mathbf{x} \|^2 - \| \mathbf{q} \|^2 + 2 \mathbf{x} \boldsymbol{\cdot} \mathbf{q}
 \\
 & = \| \mathbf{p}_i \|^2 - \| \mathbf{q} \|^2 - 2 \mathbf{x} \boldsymbol{\cdot} ( \mathbf{p}_i - \mathbf{q} ) .
\end{split}
\]
This implies
\[
\sum_{i=1}^d u_i ( F_i ( \mathbf{x} ) - F ( \mathbf{x} ) ) = 2 \mathbf{x} \boldsymbol{\cdot} \sum_{i=1}^d u_i ( \mathbf{p}_i - \mathbf{q} ) + \sum_{ i = 1 }^d u_i ( \| \mathbf{p}_i \|^2 - \| \mathbf{q} \|^2 ) = - c .
\]
By definition of \( b \) one obtains \( ( \sum_{i=1}^d u_i F_i ( \mathbf{x} ) ) + b F ( \mathbf{x} ) + c = 0 \) for all \( \mathbf{x} \in \R^{ d - 1} \), as required.
\end{proof}

The definition of \( \aleph_ \delta \)-spray was given for \( \R^d \), but it can be adapted to the space 
\[
 \mathbb{H}^d \coloneqq \R^{d - 1} \times ( 0 ; + \infty ) \subseteq \R^d
\]
as follows: an \( \aleph_ \delta \)-spray in \( \mathbb{H}^d \) is a set \( X \subseteq \mathbb{H}^d \) together with a point \( \mathbf{c} \), the center of \( X \), belonging to the hyperplane \( \R^{ d - 1} \times \setLR{0} \) of \( \R^d \), such that \( \card{ S \cap X } < \aleph_ \delta \) for any sphere \( S \) of \( \R^d \) centered in \( \mathbf{c} \).
(Observe that \( \mathbf{c} \) does not belong to \( \mathbb{H} \).)

If \( X_1 , X_2 , \dots \) are \( \aleph_ \delta \)-sprays in \( \R^d \) with centers \( \mathbf{c}_1 , \mathbf{c}_2 , \ldots \in \R^{d - 1} \times \set{0} \), then \( X_1 \cap \mathbb{H}^d , X_2 \cap \mathbb{H}^d , \dots \) are \( \aleph_ \delta \)-sprays in \( \mathbb{H}^d \) with the same centers; moreover, if \( X_1 , X_2 , \dots \) cover \( \R^d \), then \( X_1 \cap \mathbb{H}^d , X_2 \cap \mathbb{H}^d , \dots \) cover \( \mathbb{H}^d \).
Therefore we may focus on covering \( \mathbb{H}^d \) with \( \aleph_ \delta \)-sprays in \( \mathbb{H}^d \) with centers on \( \R^{ d - 1 } \times \set{0} \).

Let \( \mathbf{c}_1 , \dots , \mathbf{c}_d \) be points in general position in \( \R^d \).
These points belong to a hyperplane, and without loss of generality we may assume that \( \mathbf{c}_1 , \dots , \mathbf{c}_d \in \R^{ d - 1 } \times \set{0} \).
The set 
\[
 E^d \coloneqq \setof{ ( r_1 , \dots , r_d ) }{ \EXISTS{ \mathbf{x} \in \mathbb{H}^d } \FORALL{i \leq d} ( r_i = \| \mathbf{x} - \mathbf{c}_i \|^2 ) } 
\]
is an open subset of \( \R^d_+ \coloneqq \setof{ ( r_1 , \dots , r_d ) \in \R^d }{ \FORALL{i \leq d} ( r_i > 0 )} \), and 
\begin{equation}\label{eq:Phi}
\Phi \colon \mathbb{H}^d \to E^d , \qquad \mathbf{x} \mapsto ( \| \mathbf{x} - \mathbf{c}_1 \|^2 , \dots , \| \mathbf{x} - \mathbf{c}_d \|^2 ) 
\end{equation}
is a homeomorphism.
Every element of \( E = E^d \) determines a unique point in \( \mathbb{H}^d \), and this is the reason for retreating form \( \R^d \) to \( \mathbb{H}^d \).

The map \( \Phi \) transforms spheres centered around \( \mathbf{c}_i \) into hyperplanes of \( \R^d \) orthogonal to \( \mathbf{e}_i \), and conversely. 
To be specific---and recalling Notation~\ref{ntn:hyperplane}---if \( S \) is the sphere of \( \R^d \) centered in \( \mathbf{c}_i \) of radius \( r \), then \( \Phi [ S \cap \mathbb{H}^d ] \) is the intersection between \( E \) and the hyperplane \( H_i ( r^2 ) \) of \( \R^d \); conversely, \( \Phi^{-1} [ H_i ( r^2 ) \cap E ] = S \cap \mathbb{H}^d \).
Therefore if \( X \subseteq \mathbb{H}^d \) is an \( \aleph_ \delta \)-spray centered in \( \mathbf{c}_i \), then \( \Phi [ X ] \subseteq E \) is such that \( \card{ \Phi [ X ] \cap H_i ( r ) } < \aleph_ \delta \) for every \( r > 0 \); conversely if \( Y \subseteq E \) intersects all hyperplanes orthogonal to \( \mathbf{e}_i \) in a set of size \( < \aleph_ \delta \) then \( \Phi ^{- 1 } [ Y ] \subseteq \mathbb{H}^d \) is an \( \aleph_ \delta \)-spray centered in \( \mathbf{c}_i \).

Suppose \( \mathbf{c}_{ d + 1 } \in \R^{ d - 1 } \times \set{0} \) is distinct from \( \mathbf{c}_1 , \dots , \mathbf{c}_d \).
Let
\[
\check{ E } = \setof{ ( r_1 , \dots , r_{ d + 1 } ) \in \R^{ d + 1 }_+}{ \EXISTS{ \mathbf{x} \in \mathbb{H}^d } \FORALL{i \leq d + 1 } ( r_i = \| \mathbf{x} - \mathbf{c}_i \|^2 ) } .
\] 
Then \( \check{ \Phi } \colon \mathbb{H}^d \to \check{ E } \), \( \mathbf{x} \mapsto( \| \mathbf{x} - \mathbf{c}_1 \|^2 , \dots , \| \mathbf{x} - \mathbf{c}_{ d + 1 } \|^2 ) \), is a homeomorphism.
Applying Theorem~\ref{th:Ivan}with \( \mathbf{p}_i = \mathbf{c}_i \) and \( \mathbf{q} = \mathbf{c}_{ d + 1 } \), there are \( \mathbf{u} = ( u_1 , \dots , u_d ) \in \mathcal{U} ( \mathbf{c}_{ d + 1 } ) \setminus \set{ \mathbf{0} } \) and \( b , c \in \R \), such that 
\[ 
 P = \setof{ ( w_1 , \dots , w_{ d + 1 } ) }{ u_1 w_1 + \dots + u_d w_d + b w_{ d + 1 } + c = 0 } 
\] 
is a hyperplane of \( \R^{ d + 1 } \).
As \( \mathbf{c}_{ d + 1 } \neq \mathbf{c}_i \) for \( i \leq d \), the vector \( \mathbf{u} \) is not orthogonal to any \( \mathbf{e}_i \) for \( 1 \leq i \leq d \), and the same is true of \( P \).
Therefore the projection \( \pi \colon \R^d \times \R \to \R^d \) is a bijection between \( P \) and \( \R^d \).
Let
\[
\check{ E } = \setof{ ( r_1 , \dots , r_{ d + 1 } ) \in \R^{ d + 1 } _+ }{ ( r_1 , \dots , r_d ) \in E \AND ( r_1 , \dots , r_{ d + 1 } ) \in P } ,
\]
that is: \( \check{ E } \) is the subset of \( P \) that projects onto \( E \), and \( \pi \restriction \check{ E } \) is a homeomorphism between \( \check{ E } \) and \( E \). 
For any \( k \in \R \) the set 
\[
\begin{split}
L ( \mathbf{u} , k ) & \coloneqq \setof{ ( r_1 , \dots , r_d ) }{ u_1 r_1 + \dots + u_d r_d + b k + c = 0 } 
\\
 & = \setof{ ( r_1 , \dots , r_d ) }{ ( r_1 , \dots , r_d , k ) \in P } 
\end{split}
\] 
is a hyperplane of \( \R^d \) orthogonal to \( \mathbf{u} \), and all hyperplanes of \( \R^d \) orthogonal to \( \mathbf{u} \) are of this form.
Arguing as before, if \( X \subseteq \mathbb{H}^d \) is an \( \aleph_ \delta \)-spray centered in \( \mathbf{c}_{ d + 1 } \) then \( \Phi [ X ] \) is a subset of \( E \) that intersects every \( L ( \mathbf{u} , k ) \) in a set of size \( < \aleph_ \delta \); conversely if \( Y \subseteq E \) is such that every \( L ( \mathbf{u} , k ) \) intersects \( Y \) in \( < \aleph_ \delta \)-many points, then \( \Phi ^{ - 1 } [ Y ] \) is an \( \aleph_ \delta \)-spray centered in \( \mathbf{c}_{ d + 1 } \).

Let us summarize what we proved so far.

\begin{theorem}\label{th:transfer}
Suppose \( \mathbf{c}_1 , \dots , \mathbf{c}_d \) are in general position in \( \R^d \), and without loss of generality we may assume that they belong to \( \R^{d - 1} \times \set{0} \).
There is an open set \( \emptyset \neq E^d \subseteq \R^d \) and a homeomorphism \( \Phi \colon \mathbb{H}^d \to E^d \) that transforms any \( \aleph_ \delta \)-spray of \( \mathbb{H}^d \) centered in \( \mathbf{c}_i \) into a subset of \( E^d \) that intersects any hyperplane orthogonal to \( \mathbf{e}_i \) in a set of size \( < \aleph_ \delta \), and conversely. 

Let \( \mathbf{c} \in \R^{d - 1} \times \set{0} \) be distinct from the \( \mathbf{c}_i \)s, and let \( ( u_1 , \dots , u_d ) \in \mathcal{U} ( \mathbf{c} ) \setminus \set{ \mathbf{0} } \).
Then \( \Phi \) maps any \( \aleph_ \delta \)-spray centered in \( \mathbf{c} \) into a subset of \( E \) that intersects any hyperplane orthogonal to \( \mathbf{u} = ( u_1 , \dots , u_d )\) in a set of size \( < \aleph_ \delta \), and conversely.
\end{theorem}

The following picture when \( d = 2 \) may help to visualize the previous construction.
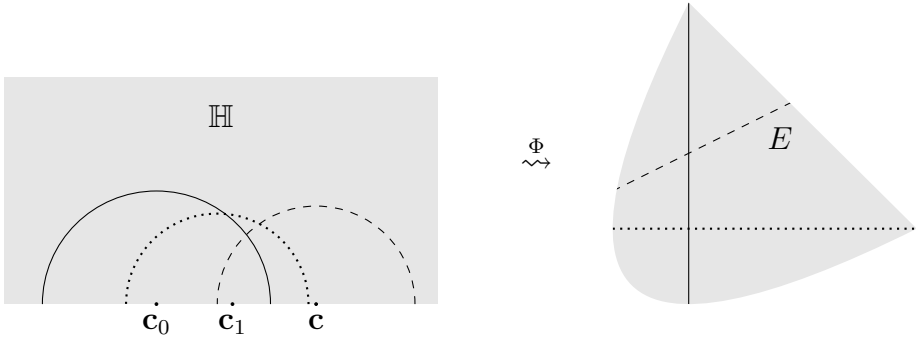
\[
\begin{tikzpicture}
\fill[gray!20] (-2,0) -- (3.7,0)--(3.7,3)--(-2,3);
\draw (1.5,0) arc [start angle=0, end angle=180, radius=1.5cm];
\draw[thick,dotted] (2,0) arc [start angle=0, end angle=180, radius=1.2cm];
\draw[dashed] (3.4,0) arc [start angle=0, end angle=180, radius=1.3cm];
\filldraw (0,0) circle [radius=0.5pt] node[anchor=north] {$\mathbf{c}_0$};
\filldraw (1,0) circle [radius=0.5pt] node[anchor=north] {$\mathbf{c}_1$};
\filldraw (2.1,0) circle [radius=0.5pt] node[anchor=north] {$\mathbf{c}$};
\node at (0.85 , 2.5) {\( \mathbb{H} \)};
\node at (5, 2) {\( \overset{\Phi}{\leadsto} \)};
\fill[gray!20] (7,4) .. controls (5,0) and (6,-1) .. (10,1) -- (7,4);
\clip (7,4) .. controls (5,0) and (6,-1) .. (10,1) -- (7,4);
\node at (8.2, 2.2) {\( E \)};
\draw[] (7,0) -- (7,4);
\draw[thick,dotted] (6,1) -- (10,1);
\draw[dashed] (6,1.5) -- (11,4);
\end{tikzpicture}
\]

We can now sketch Schmerl's argument that if the plane is covered with three sprays with collinear centers, then \( \CH \) holds.
First of all we may assume that the centers lie on the \( x \)-axis so that the three sprays cover \( \mathbb{H}^2 \).
Then \( E^2 \) can be covered with three sets \( A_1 , A_2 , A_3 \) such that the intersection of \( A_i \) with any line orthogonal to some vector \( \mathbf{u}_i \) is finite (\( i = 1 , 2 , 3 \)), and by a theorem of Bagemihl~\cite{Bagemihl:1961ve} this implies \( \CH \).
Our strategy is to replace lines in \( \R^2 \) with hyperplanes in \( \R^d \), and this is the topic of the next section.

\section{Hyperplane sections}\label{sec:Hyperplane-sections}
Over the last century several results were obtained, establishing connections between the size of the continuum and elementary properties of the euclidean spaces.
The first such result is Sierpiński's theorem form 1919, asserting that \( \CH \) is equivalent to a covering of the plane with two sets such that the intersection of the first set with any vertical line is countable and the intersection of every horizontal line with the second set is countable (see Theorem~\ref{th:Sierpinski}\ref{th:Sierpinski-a}).
In~\cite{Sierpinski:1951aa} Sierpiński sharpened his previous result by replacing ``countable'' with ``finite'', but at the cost of increasing the dimension: \( \CH \) is equivalent to a decomposition \( A_1 , A_2 , A_3 \) of \( \R^3 \) such that the intersection of any line with direction \( \mathbf{e}_i \) with \( A_i \) is finite (Theorem~\ref{th:Sierpinski}\ref{th:Sierpinski-b}), and this was quickly generalized by Kuratowski to higher dimensions~\cite{Kuratowski:1951aa}: \( 2^{\aleph_0 } \leq \aleph_n \) if and only if there is a decomposition \( A_1 , \dots , A_{ n + 2} \) of \( \R^{ n + 2 } \) such that every line parallel to \( \mathbf{e}_i \) has finite intersection with \( A_i \).
In the sixties of the last century, Bagemihl and Davies showed that Kuratowski's result could be proved for \( \R^2 \) by taking intersections with lines of prescribed directions~\cite{Bagemihl:1960ti,Bagemihl:1961ve,Davies:1962ly,Davies:1963zr,Bagemihl:1968up}.
A line is a hyperplane in \( \R^2 \), so Sierpiński's result form 1919 could be stated as: if \( d = 2 \) then \( \CH \) is equivalent to a decomposition of \( \R^d \) into \( A_1 , \dots , A_d \) such that every hyperplane orthogonal to \( \mathbf{e}_i \) has countable intersection with \( A_i \).
This result holds for all \( d \geq 2 \); Sierpiński states and proves it for \( d = 3 \) in~\cite{Sierpinski:1951aa}, see Corollary~\ref{cor:Sierpinski51} below, and observes that the result for \( \R^3 \) is false if ``countable'' is replaced by ``finite''.
Erd\H{o}s, Jackson, and Mauldin prove that \( \CH \) is equivalent to \( \R^3 \) being decomposable in five pieces \( A_1 , \dots , A_5 \) so that every plane orthogonal to \( \mathbf{u}_i \) has finite intersection with \( A_i \), where \( \mathbf{u}_1 , \dots , \mathbf{u}_5 \) are in general position in \( \R^3 \), but no decomposition exists if we just allow four vectors and four sets~\cite[Corollary 6, Lemma 1]{Erdos:1994yq}. 
In that paper it is stated (but not proved) that an analogous result holds if \( \CH \) is weakened to \( 2^{\aleph_0} \leq \aleph_n \).
Since we need a detailed analysis of the positions of the various (hyper)planes we will state and prove these results in full generality below.
This section is devoted to the following

\begin{problem}\label{prob:hyperplanes}
Given \( \mathbf{u}_1 , \dots , \mathbf{u}_n \) distinct, non-zero vectors of \( \R^d \), what conditions on the cardinality of \( \R \) are equivalent to the existence of \( A_1 , \dots , A_n \) covering \( \R^d \) such that each \( A_i \cap H_{ \mathbf{u}_i } ( \mathbf{p} ) \) is finite (or countable)?
\end{problem}

First of all we need a result allowing us to replace the standard basis of \( \R^d \) with any basis.

\begin{lemma}\label{lem:changecoordinatestogetorthogonalplanes}
Suppose \( \mathbf{u}_1 , \dots , \mathbf{u}_d \in \R^d \) are linearly independent, and \( d \geq 2 \).
There is a linear isomorphism \( T \colon \R^d \to \R^d \) that maps every hyperplane orthogonal to \( \mathbf{u}_i \) to a hyperplane orthogonal to \( \mathbf{e}_i \), and conversely, for any \( i \leq d \).
\end{lemma}

\begin{proof}
Any linear \( T \) preserves translations, so it is enough to consider hyperplanes passing through the origin.
In other words it is enough to find an injective \( T \) such that \( T [ H_{\mathbf{u}_i } ( \mathbf{0} ) ] \subseteq H_{\mathbf{e}_i } ( \mathbf{0} ) \).
For \( 1 \leq i \leq d \) let \( \mathbf{u}_i = ( u_{i , 1} , \dots , u_{i , d} ) \).
Consider the \( d \times d \)-matrix \( M = ( u_{ i , j } )^{\mathrm{t}}\) that is the transpose of the matrix of the linear transformation such that \( \mathbf{e}_i \mapsto \mathbf{u}_i \) for \( i = 1 , \dots , d \).
The matrix \( M \) is invertible, as the \( \mathbf{u}_i \)s are linearly independent, so the transformation \( T \) induced by \( M \) is an isomorphism.
If \( \mathbf{x} = ( x_1 , \dots , x_d ) \in H_{\mathbf{u}_i} ( \mathbf{0} ) \), then \( \mathbf{u}_i \boldsymbol{\cdot} \mathbf{x}^{\mathrm{t}}= \sum _{ j =1 }^d u_{i , j \, } x_j = 0 \), and this is the \( i \)-th row of \( M \mathbf{x}^{\mathrm{t}} \).
We obtain that \( M \mathbf{x}^{\mathrm{t}} \in H_{\mathbf{e}_i} ( \mathbf{0}) \).
Therefore \( T [ H_{\mathbf{u}_i} ( \mathbf{0} ) ] \subseteq H_{\mathbf{e}_i} ( \mathbf{0} ) \) for all \( i < d \), and this is what we had to prove.
\end{proof}

\begin{lemma}\label{lem:setofdifferenceinabeliangroup}
Let \( \varepsilon > 0 \).
For all \( X \subseteq \R \) with \( \aleph_0 \leq \card{X} < \kappa \leq 2^{\aleph_0} \), there is \( S \subseteq ( - \varepsilon ; \varepsilon ) \) such that \( \card{S} = \kappa \) and \( ( X - X ) \cap ( S - S ) = \set{ 0 } \).
\end{lemma}

\begin{proof}
Let \( \tilde{X} \) be the subgroup of \( ( \R , + ) \) generated by \( X \).
Let \( T \subseteq \R \) be a transversal for the quotient, i.e.~a set picking exactly one element from each coset of \( \R / \tilde{X} \).
Then \(  2^{\aleph_0} = \card{ \R / \tilde{X} } = \card{T} = \card{T \cap ( - \varepsilon ; \varepsilon )} \).
Let \( S \subseteq T \cap ( - \varepsilon ; \varepsilon ) \) be of size \( \kappa \).

If \( s_1 - s_2 = x_1 - x_2 \) for some \( s_1 , s_2 \in S \) and \( x_1 , x_2 \in X \), since \( x_1 - x_2 \in \tilde{X} \) this means that the cosets \( s_1 + \tilde{X} \) and \( s_2 + \tilde{X} \) are the same, so \( s_1 = s_2 \) by definition of \( S \).
Therefore \( ( S - S ) \cap ( X - X ) = \set{ 0 } \). 
\end{proof}

The next result asserts that if the sets \( A_i \subseteq \R^d \) have small intersection with the hyperplanes orthogonal to \( \mathbf{u}_i \), and if \( 2^{\aleph_0} \) is large enough, then there is a set \( Z \subseteq \R^d \) such that no translate of it can be covered by the \( A_i \)s---in particular the \( A_i \)s do not cover \( \R^d \).

\begin{theorem}\label{th:EJM-Rd}
Let \( \delta \in \On \), \( d \geq 3 \) and let \( N = 2 ( d - 1 ) \).
Suppose \( \mathbf{u}_1 , \dots , \mathbf{u}_N \) are distinct, non-zero vectors of \( \R^d \).
Suppose \( A_1 , \dots , A_N \) are subsets of \( \R^d \) such that for all \( \mathbf{p} \in \R^d \), and for all \( 1 \leq i \leq N \)
\[
\card{ H_{\mathbf{u}_i} ( \mathbf{p} ) \cap A_i } < \aleph_ \delta . 
\]
If \( 2^{\aleph_0 } > \aleph_{ \delta } \) then for every \( \varepsilon > 0 \) there is \( Z \subseteq ( - \varepsilon ; \varepsilon ) ^ d \) of size \( \aleph_{ \delta + 1 } \) such that 
\[
\FORALL{ \mathbf{p} \in \R^d } ( \mathbf{p} + Z \nsubseteq \textstyle \bigcup_{i = 1 }^N A_i ) . 
\]
\end{theorem}

Before proving this let us draw a few corollaries.
With the same notation as before:

\begin{theorem}\label{th:EJM-Rd-1}
Let \( \delta \in \On \), \( d \geq 3 \) and \( N = 2 ( d - 1 ) \).
Suppose \( \mathbf{u}_1 , \dots , \mathbf{u}_N \) are non-zero vectors of \( \R^d \).
Suppose \( D \subseteq \R^d \) is such that \( \Int ( D ) \neq \emptyset \).
If \( A_1 , \dots , A_N \) cover \( D \) and for all \( \mathbf{p} \in \R^d \) and for all \( 1 \leq i \leq N \)
\[
\card{ H_{\mathbf{u}_i} ( \mathbf{p} ) \cap A_i } < \aleph_ \delta .
\]
Then \( 2^{\aleph_0} \leq \aleph_{ \delta } \).
\end{theorem}

\begin{proof}
Towards a contradiction, suppose \( 2^{\aleph_0} \geq \aleph_{ \delta + 1 } \).
Let \( \varepsilon > 0 \) be small enough so that \( \mathbf{p} + ( - \varepsilon ; \varepsilon )^d \subseteq D \) for some \( \mathbf{p} \in \R^d \), and, towards a contradiction, suppose \( A_1 , \dots , A_N \) are as in the statement.
Let \( Z \subseteq ( - \varepsilon ; \varepsilon )^d \) be as in Theorem~\ref{th:EJM-Rd}: then \( \mathbf{p} + Z \) is contained in \( D \) but on other hand it is not contained in \( A_1 \cup \dots \cup A_N \), a contradiction.
\end{proof}

The presence of the set \( D \) in the statement of Theorem~\ref{th:EJM-Rd-1} may seem unwarranted right now, but it will be crucial for the results in Section~\ref{sec:The-main-results}.
For the time being the reader can safely replace \( D \) with \( \R^d \) without losing much.

\begin{theorem}\label{th:EJM-Rd-2}
For all \( d \geq 3 \) and all \( D \subseteq \R^d \) such that \( \emptyset \neq \Int ( D ) \), the following are equivalent:
\begin{enumerate}[label={\upshape (\alph*)}]
\item\label{th:EJM-Rd-2-a}
\( \CH \).
\item\label{th:EJM-Rd-2-b}
For all non-zero vectors \( \mathbf{u}_1 , \dots , \mathbf{u}_d \) spanning \( \R^d \), there are \( A_1 , \dots , A_d \) covering \( D \) such that \( H_{ \mathbf{u}_i } ( \mathbf{p} ) \cap A_i \) is countable for all \( \mathbf{p} \in \R^d \) and \( 1 \leq i \leq d \).
\item\label{th:EJM-Rd-2-c}
There are non-zero vectors \( \mathbf{u}_1 , \dots , \mathbf{u}_{ 2 d - 2 } \) and there are \( A_1 , \dots , A_{ 2 d - 2 } \) covering \( D \) such that \( H_{ \mathbf{u}_i } ( \mathbf{p} ) \cap A_i \) is countable for all \( \mathbf{p} \in \R^d \) and \( 1 \leq i \leq 2 d - 2 \).
\end{enumerate} 
\end{theorem}

\begin{proof}
The implication \ref{th:EJM-Rd-2-a}\( \IMPLIES \)\ref{th:EJM-Rd-2-b} follows from Theorem~\ref{th:EJM2} and Example~\ref{xmp:well-spread}, while \ref{th:EJM-Rd-2-b}\( \IMPLIES \)\ref{th:EJM-Rd-2-c} is trivial---take \( A_{ d + 1 } = \dots = A_{ 2 d - 2 } = \emptyset \).
The implication \ref{th:EJM-Rd-2-c}\( \IMPLIES \)\ref{th:EJM-Rd-2-a} follows from Theorem~\ref{th:EJM-Rd-1} when \( \delta = 1 \).
\end{proof}

\begin{corollary} \label{cor:Sierpinski51}
For all \( d \geq 3 \), \( \CH \) is equivalent to the existence of \( A_1 , \dots , A_d \) covering \( \R^d \) such that \( H_i ( x ) \cap A_i \) is countable, for all \( x \in \R \) and all \( 1 \leq i \leq d \).
\end{corollary}

Cantor's theorem says that \( 2^{\aleph_0} > \aleph_{ \delta } \) if \( \delta = 0 \), so when \( d = 3 \) Theorem~\ref{th:EJM-Rd} implies the following result, which is Lemma 1 of~\cite{Erdos:1994yq}:

\begin{corollary} \label{cor:EJM-Rd-2}
Suppose \( \mathbf{u}_1 , \dots , \mathbf{u}_4 \) are non-zero vectors of \( \R^3 \).
There are no \( A_1 , \dots , A_4 \) covering \( \R^3 \) such that \( H_{\mathbf{u}_i} ( \mathbf{p} ) \cap A_i \) is finite, for all \( \mathbf{p} \in \R^3 \) and all \( 1 \leq i \leq 4 \).
\end{corollary}

\begin{remark}
Corollary~\ref{cor:EJM-Rd-2} is a negative result, asserting that \( \R^3 \) is not the union of four sets such that each intersects any plane orthogonal to a given vector in a finite set.
But looking at the proof of Theorem~\ref{th:EJM-Rd}, it could be recast in a positive way as:
\begin{quote}
\textit{Suppose \( \Bbbk \) is an infinite field. 
If there are non-zero vectors \( \mathbf{u}_1 , \dots , \mathbf{u}_4 \) of \( \Bbbk^3 \) and sets \( A_1 , \dots , A_4 \) covering \( \Bbbk^3 \) such that every hyperplane orthogonal to \( \mathbf{u}_i \) has finite intersection with \( A_i \), then \( \card{ \Bbbk } = \aleph_0 \).}
\end{quote}
\end{remark}

\begin{proof}[Proof of Theorem~\ref{th:EJM-Rd}]
Let \( A_1 , \dots , A_N \) and \( \varepsilon > 0 \) be as in the statement.
We consider two cases, depending whether the \( \mathbf{u}_i \)s span \( \R^d \).

\smallskip

\textbf{Case 1:} the vectors \( \mathbf{u}_1 , \dots , \mathbf{u}_N \) do not span \( \R^d \).

Let \( \mathbf{v} \) be a non-zero vector such that \( \mathbf{v} \boldsymbol{\cdot} \mathbf{u}_k = 0 \) for all \( 1 \leq k \leq N \), and let \( V = \setof{ r \mathbf{v} }{ r \in \R } \).
Since \( \card{V} = 2^{\aleph_0 } \geq \aleph_{ \delta + 1 } \) we can take \( Z \subseteq V \cap ( - \varepsilon ; \varepsilon )^d \) of size \( \aleph_{ \delta + 1 } \).
Since \( \mathbf{p} + V \subseteq H_{ \mathbf{u}_i } ( \mathbf{p} ) \) we have that \( \card{ ( \mathbf{p} + V ) \cap \bigcup_{ i = 1 }^N A_i } < \aleph_ \delta \), and hence \( \mathbf{p} + Z \nsubseteq \bigcup_{ i = 1 }^N A_i \).

\smallskip

\textbf{Case 2:} the vectors \( \mathbf{u}_1 , \dots , \mathbf{u}_N \) span \( \R^d \).

Without loss of generality we may assume that \( ( \mathbf{u}_1 , \dots , \mathbf{u}_d ) \) is a basis of \( \R^d \).
We first prove the result under the additional assumption that 
\[ 
( \mathbf{u}_1 , \dots , \mathbf{u}_d ) \text{ is the standard basis } ( \mathbf{e}_1 , \dots , \mathbf{e}_d ) .
\]
Let
\begin{align*}
U & = \Span \set{ \mathbf{u}_{ d + 1 } , \dots , \mathbf{u}_N } , & U_j & = \Span \set{ \mathbf{e}_j , \mathbf{u}_{ d + 1 } , \dots , \mathbf{u}_N } \text{, for } 1 \leq j \leq d .
\end{align*}
As \( \dim U \leq d - 2 \) then \( \dim U_j \leq d - 1 \), so there is a non-zero vector orthogonal to \( U_j \).

\begin{claim}
There is \( 1 \leq j \leq d \) such that there is \( \mathbf{v} \neq \mathbf{0} \) orthogonal to \( U_j \), and \( \mathbf{v} \) not collinear to any \( \mathbf{e}_i \).
\end{claim}

\begin{proof}[Proof of the Claim]
Suppose \( \dim U_j \leq d - 2 \) for some \( j \), and that every \( \mathbf{v} \neq \mathbf{0} \) orthogonal to \( U_j \) is collinear to some \( \mathbf{e}_i \).
As the orthogonal complement of \( U_j \) has dimension \( \geq 2 \), there are \( \mathbf{e}_{i_1} , \mathbf{e}_{i_2} \) orthogonal to \( U_j \), and hence \( \mathbf{v} = \mathbf{e}_{i_1} + \mathbf{e}_{i_2} \) is orthogonal to \( U_j \) and is not collinear to any vector of the standard basis, against our assumption.

Therefore we may assume that \( \dim U_j = d - 1 \) for all \( j \), and hence \( \dim U = d - 2 \).
Towards a contradiction, suppose that for every \( j \in \set{ 1 , \dots , d } \) there is a (necessarily unique) \( j^* \neq j \) such that \( \mathbf{e}_{j^*} \) is orthogonal to \( U_j \).
The map \( j \mapsto j^* \) is not constant, so fix \( i \neq j \) such that \( i^* \neq j^* \).
As \( d \geq 3 \) pick \( k \in \set{ 1 , \dots , d } \) distinct from \( i^* , j^* \).
Both vectors \( \mathbf{e}_{i^*} \) and \( \mathbf{e}_{j^*} \) are orthogonal to \( U \) and to \( \mathbf{e}_k \), so both are orthogonal to \( U_k \), which is impossible as \( \dim U_k = d - 1 \).
\end{proof}

Fix \( j \) and \( \mathbf{v} = ( a_1 , \dots , a_d ) \) as in the Claim.
By a re-indexing, if needed, we may assume that \( j = d \), that is \( a_d = 0 \), and that \( a_1 , a_2 \) are non-zero.
By rescaling we may assume that 
\[
 \mathbf{v} = ( 1 , a_2 , \dots , a_{ d - 1} , 0 ) , \qquad a_2 \neq 0
\] 
and that
\begin{equation}\label{eq:th:EJM-Rd-0}
 d \leq k \leq N \IMPLIES \mathbf{v} \boldsymbol{\cdot} \mathbf{u}_k = 0 . 
\end{equation}
Let \( \nu \coloneqq \max \setLR{ 1 , \card{ a_2 } , \dots , \card{ a_{ d - 1 } } } \).
For \( 1 \leq i \leq d \) let \( X_i \subseteq ( - \varepsilon / 2 ; \varepsilon / 2 ) \) such that \( \card{ X_1 } = \aleph_{ \delta + 1} \) and \( \card{ X_i } = \aleph_ \delta \) for all \( 2 \leq i \leq d \). 
By Lemma~\ref{lem:setofdifferenceinabeliangroup}, let \( S \subseteq ( - \frac{ \varepsilon } { 2 \nu } ; \frac{ \varepsilon } { 2 \nu } ) \) be of size \( \aleph _{ \delta + 1 } \) such that \( ( S - S ) \cap ( a_2^{-1}X_2 - a_2^{-1} X_2 ) = \set{0} \), and let 
\[
 V = \setof{ s \mathbf{v} }{ s \in S } .
\]
By~\eqref{eq:th:EJM-Rd-0}, if \( \mathbf{q} \in \R^d \) and \( s \in S \) then 
\begin{equation}\label{eq:th:EJM-Rd-1}
d \leq k \leq N \IMPLIES \mathbf{q} + V \subseteq H_{\mathbf{u}_k} ( \mathbf{q} + s \mathbf{v} ) .
\end{equation}
As \( V \subseteq ( - \varepsilon / 2 ; \varepsilon / 2 )^d \) it follows that 
\[
Z \coloneqq V + \textstyle \prod_{i = 1}^d X_i
\]
is a subset of \( ( - \varepsilon ; \varepsilon )^d \), and it is of cardinality \( \aleph_{ \delta + 1 } \).
Fix \( \mathbf{p}= ( p_1 , \dots , p_d ) \in \R^d \) and, towards a contradiction, suppose that \( \mathbf{p} + Z \subseteq \bigcup_{i = 1 }^N A_i \).
Observe that \( Z = \bigcup_{ x \in X_1 } Z_x \), where 
\[ 
Z_x \coloneqq V + ( \set{x} \times \textstyle \prod_{i = 2 }^d X_i ) .
\]

\begin{claim}\label{cl:th:EJM-Rd-1}
The sets \( Z_x \) are pairwise disjoint.
\end{claim}

\begin{proof}
If \( Z_{x'} \cap Z_{x''} \neq \emptyset \), then there are \( s' , s'' \in S \), \( x_i' , x_i'' \in X_i \) for \( 2 \leq i \leq d \) such that 
\[
s' \mathbf{v} + ( x ' , x'_2 , \dots , x'_d ) = s'' \mathbf{v} + ( x '' , x_2 '' , \dots , x''_d ) .
\]
The \( 2 \)-nd component yields that \( s' a_2 + x_2' = s'' a_2 + x_2 '' \), and hence 
\[
s' - s'' = a_2^{-1} ( x_2 '' - x_2' ) \in ( S - S ) \cap ( a_2^{-1} X_2 - a_2^{-1} X_2 ) = \setLR{0} .
\]
It follows that \( s' = s'' = s \), and hence 
\[
s \mathbf{v} + ( x ' , x_2 ' , \dots , x'_d ) = s \mathbf{v} + ( x '' , x_2 '' , \dots , x''_d ) .
\]
Looking at the \( 1 \)-st component we obtain \( s + x ' = s + x '' \) and hence \( x ' = x '' \).
\end{proof}

Summarizing: 
\begin{equation}\label{eq:th:EJM-Rd-2}
 \mathbf{p} + Z = \bigcup_{ x \in X_1 } \mathbf{p} + Z_x \text{, and } x' \neq x'' \IMPLIES \mathbf{p} + Z_{x'} \cap \mathbf{p} + Z_{x''} = \emptyset .
\end{equation}

\begin{claim}\label{cl:th:EJM-Rd-2}
\( \card{ A_d \cap ( \mathbf{p} + Z ) } < \aleph_{ \delta + 1 } \).
\end{claim}

\begin{proof}
As the last component of \( \mathbf{v} \) is \( 0 \), the set \( V \) does not contribute to \( Z \) on the \( d \)-th coordinate, so for all \( w \in \R \)
\[ 
\begin{split}
( \mathbf{p} + Z ) \cap H_d ( p_d + w ) & = ( \mathbf{p} + \textstyle \prod_{ i = 1 }^d X_i ) \cap H_d ( p_d + w ) 
\\
 & = \mathbf{p} + \bigl ( ( \textstyle \prod_{ i = 1 }^d X_i ) \cap H_d ( w ) \bigr )
\end{split}
\]
and this set is \( \mathbf{p} + ( \prod_{i = 1 }^{ d - 1} X_i ) \times \setLR{ w } \) if \( w \in X_d \), and it is empty if \( w \notin X_d \).
Therefore \( \mathbf{p} + Z \subseteq \bigcup_{w \in X_d } H_d ( p_d + w ) \) and 
\[
A_d \cap ( \mathbf{p} + Z ) \subseteq \textstyle \bigcup_{ w \in X_d } A_d \cap H_d ( p_d + w ) .
\]
By assumption \( \card{ A_d \cap H_d ( p_d + w ) } < \aleph_{ \delta } \) for any \( w \in \R \), and \( \card{ X_d } = \aleph_ \delta \), so \( \card{ A_d \cap ( \mathbf{p} + Z ) } \leq \aleph_ \delta < \aleph_{ \delta + 1 } \).
\end{proof}

As \( \card{ X_1 } = \aleph_ { \delta + 1} \) and by~\eqref{eq:th:EJM-Rd-2}, there is \( \bar{x}_1 \in X_1 \) such that \( \mathbf{p} + Z_{\bar{x}_1 } \) is disjoint from \( A_d \).
This implies that 
\[ 
\mathbf{p} + Z_{ \bar{x}_1} \subseteq A_1 \cup \dots \cup A_{ d - 1 } \cup \textstyle \bigcup_{ k = d + 1 }^N A_k .
\]

\begin{claim}\label{cl:th:EJM-Rd-3}
There is an \( \bar{s} \in S \) such that \( \mathbf{p} + \bar{s} \mathbf{v} + ( \set{ \bar{x}_1 } \times \prod_{ i = 2 }^d X_i ) \) is disjoint from \( \bigcup_{ k = d + 1 }^N A_k \).
\end{claim}

\begin{proof}
Towards a contradiction, suppose that for all \( s \in S \) and all \( 2 \leq i \leq d \) there are \( x_i ( s ) \in X_i \) such that 
\[ 
 \mathbf{p} + s \mathbf{v} + ( \bar{x}_1 , x_2 ( s ) , \dots , x_d ( s ) ) \in \textstyle \bigcup_{ k = d + 1 }^N A_k .
\]
As \( \card{X_i } = \aleph_ \delta \) for \( 2 \leq i \leq d \), there are \( x_i' \in X_i \) and \( S' \subseteq S \) of size \( \aleph _{ \delta + 1 } \) such that \( x_i ( s ) = x_i ' \) for all \( s \in S ' \). 
Therefore 
\[
\FORALL{ s \in S '} \bigl ( \mathbf{p} + s \mathbf{v} + ( \bar{x}_1 , x_2 ' , \dots , x_d ' ) \in \textstyle\bigcup_{ k = d + 1 }^N A_k \bigr ) .
\] 
As \( \card{S'} = \aleph_{ \delta + 1 } \) there is \( S^* \subseteq S' \) of size \( \aleph_{ \delta + 1 } \), and there is \( k \) with \( d < k \leq N \) such that 
\begin{equation}\label{eq:th:EJM-Rd-3}
\FORALL{ s \in S ^*} \bigl ( \mathbf{p} + s \mathbf{v} + ( \bar{x}_1 , x_2 ' , \dots , x_d ' ) \in A_k \bigr ) . 
\end{equation}
Fix \( s^* \in S^* \) and let 
\[
 Q_k \coloneqq H_{\mathbf{u}_k} \bigl ( \mathbf{p} + s^* \mathbf{v} + ( \bar{x}_1 , x_2 ' , \dots , x_d ' ) \bigr ) .
\]
By~\eqref{eq:th:EJM-Rd-1} with \( \mathbf{q} = \mathbf{p} + ( \bar{x}_1 , x_2 ' , \dots , x_d ' ) \) we have that 
\[ 
\FORALL{s \in S^*}\bigl ( \mathbf{p} + s \mathbf{v} + ( \bar{x}_1 , x_2 ' , \dots , x_d ' ) \in Q_k \bigr ) .
\]
Therefore by~\eqref{eq:th:EJM-Rd-3} \( \mathbf{p} + s \mathbf{v} + ( \bar{x}_1 , x_2 ' , \dots , x_d ' ) \in A_k \cap Q_k \) for all \( s \in S^* \), and this is a contradiction, since \( \card{ A_k \cap Q_k } < \aleph_ \delta \), and yet it must contain \( \aleph_{ \delta + 1 } \) points.
\end{proof}

Let \( \bar{s} \) be as in Claim~\ref{cl:th:EJM-Rd-3}.
Then \( \card{ \mathbf{p} + \bar{s} \mathbf{v} + ( \setLR{ \bar{x}_1 } \times \prod_{ i = 2 }^d X_i ) } = \aleph_ \delta \) and 
\begin{equation}\label{eq:th:EJM-Rd-4}
 W_1 \coloneqq \mathbf{p} + \bar{s} \mathbf{v} + ( \setLR{ \bar{x}_1 } \times \textstyle \prod_{ i = 2 }^d X_i ) \subseteq A_1 \cup \dots \cup A_{ d - 1 } .
\end{equation}
As \( W_1 \) is included in \( H_1 ( p_1 + \bar{s} + \bar{x}_1 ) \), and \( \card{ H_1 ( p_1 + \bar{s} + \bar{x}_1 ) \cap A_1 } < \aleph_ \delta \), then \( \card{ W_1 \cap A_1 } < \aleph_ \delta \). 
As \( \card{ X_2 } = \aleph_ \delta \), there is some \( \bar{x}_2 \in X_2 \) such that the set \( W_2 \coloneqq \mathbf{p} + \bar{s} \mathbf{v} + ( \set{ \bar{x}_1 } \times \set{ \bar{x}_2 } \times \prod_{ i = 3 }^d X_i) \) is disjoint from \( A_1 \), and hence it is contained in \( A_2 \cup \dots \cup A_{d - 1 } \) and in the hyperplane \( H_2 ( p_2 + \bar{s} a_2 + \bar{x}_2 ) \). 
As before \( \card{ W_2 \cap A_2 } < \aleph_ \delta \). 
Repeating this argument we construct \( \bar{x}_3 \in X_3 \), \ldots , \( \bar{x}_d \in X_d \) such that 
\[
\mathbf{p} + \bar{s} \mathbf{v} + ( \bar{x}_1 , \dots , \bar{x}_d ) \notin A_1 \cup \dots \cup A_{ d - 1 }
\]
against~\eqref{eq:th:EJM-Rd-4}.
This concludes the proof, assuming that \( \mathbf{u}_1 , \dots , \mathbf{u}_d \) is the standard basis.

If \( \mathbf{u}_1 , \dots , \mathbf{u}_d \) is an arbitrary basis of \( \R^d \), by Lemma~\ref{lem:changecoordinatestogetorthogonalplanes} there is a linear injective transformation \( T \colon \R^d \to \R^d \) that maps every hyperplane orthogonal to \( \mathbf{u}_i \) to a hyperplane orthogonal to \( \mathbf{e}_i \), for \( 1 \leq i \leq d \).
The transformation \( T \) maps parallel hyperplanes to parallel hyperplanes, so for \( d \leq k < N \) let \( \bar{\mathbf{u}}_k \) be a vector such that \( T \) maps every hyperplane orthogonal to \( \mathbf{u}_k \) to a hyperplane orthogonal \( \bar{\mathbf{u}}_k \).
For \( 1 \leq k \leq d \) set \( \bar{\mathbf{u}}_k \) to be \( \mathbf{e}_k \).
As \( T^{- 1 } \) is continuous, pick \( \bar{ \varepsilon } > 0 \) small enough so that \( T^{- 1 } [ ( - \bar{ \varepsilon } ; \bar{ \varepsilon } )^d ] \subseteq ( - \varepsilon ; \varepsilon )^d \).
Arguing as above there is a \( \bar{Z} \subseteq ( - \bar{ \varepsilon } ; \bar{ \varepsilon } )^d \) such that \( \mathbf{p} + \bar{Z} \nsubseteq \bigcup_{ i = 1 }^N T [ A_i ] \) for all \( \mathbf{p} \in \R^d \).
Then \( Z = T^{-1} [ \bar{Z} ] \subseteq ( - \varepsilon ; \varepsilon )^d \) is such that \( \mathbf{p} + Z \nsubseteq \bigcup_{ i = 1 }^N A_i \) for all \( \mathbf{p} \in \R^d \).
\end{proof}

\begin{theorem}\label{th:hyperplanes=>bound}
Let \( \delta \in \On \), \( d \geq 3 \), \( n \geq 0 \), and let \( N = ( n + 1 ) ( d - 1 ) + 1 \).
Suppose \( \mathbf{u}_1 , \dots , \mathbf{u}_N \) are distinct, non-zero vectors of \( \R^d \), and that \( A_1 , \dots , A_ N \subseteq \R^d \) are such that for all \( \mathbf{p} \in \R^d \)
\[
 \FORALL{ 1 \leq k \leq N} \left ( \card{ H_{\mathbf{u}_k } ( \mathbf{p} ) \cap A_k } < \aleph_ \delta \right ) .
\]
If \( 2^{\aleph_0} > \aleph_{ \delta + n } \) then for every \( \varepsilon > 0 \) there is \( Z_{ n , \varepsilon } \subseteq ( - \varepsilon ; \varepsilon )^d \) of size \( \aleph_{ \delta + n + 1} \) such that 
\[ 
 \FORALL{ \mathbf{p} \in \R^d } \left ( \mathbf{p} + Z_{ n , \varepsilon } \nsubseteq \textstyle \bigcup_{ k = 1 }^N A_k \right ) . 
\]
\end{theorem}

Let us draw some consequences from Theorem~\ref{th:hyperplanes=>bound}.

\begin{theorem}\label{th:hyperplanes<=>bound}
Let \( \delta \in \On \), let \( n \geq 1 \), and let \( d \geq 3 \).
Let also \( ( n + 1 ) ( d - 1 ) < N \leq ( n + 2 ) ( d - 1 ) \).
For all \( D \subseteq \R^d \) such that \( \Int ( D ) \neq \emptyset \), the following are equivalent:
\begin{enumerate}[label={\upshape (\alph*)}]
\item\label{th:hyperplanes<=>bound-a}
\( 2^{\aleph_0 } \leq \aleph_{ \delta + n } \).
\item\label{th:hyperplanes<=>bound-b}
For all \( \mathbf{u}_1 , \dots , \mathbf{u}_{ ( n + 1 ) ( d - 1 ) + 1 } \) in general position there are \( A_1 , \dots , A_{ ( n + 1 ) ( d - 1 ) + 1 } \) covering \( D \) such that \( \card{ A_k \cap H_{ \mathbf{u}_k } ( \mathbf{p} ) } < \aleph_ \delta \) for all \( \mathbf{p} \in \R^d \) and all \( 1 \leq k \leq ( n + 1 ) ( d - 1 ) + 1 \).
\item\label{th:hyperplanes<=>bound-c}
For all distinct non-zero vectors \( \mathbf{u}_1 , \dots , \mathbf{u}_N \) such that \( ( n + 1 ) ( d - 1 ) + 1 \) of them are in general position, there are \( A_1 , \dots , A_N \) covering \( D \) such that \( \card{ A_k \cap H_{ \mathbf{u}_k } ( \mathbf{p} ) } < \aleph_ \delta \) for all \( \mathbf{p} \in \R^d \) and all \( 1 \leq k \leq N \).
\item\label{th:hyperplanes<=>bound-d}
There are distinct non-zero vectors \( \mathbf{u}_1 , \dots , \mathbf{u}_N \) and there are \( A_1 , \dots , A_N \) covering \( D \) such that \( \card{ A_k \cap H_{ \mathbf{u}_k } ( \mathbf{p} ) } < \aleph_ \delta \) for all \( \mathbf{p} \in \R^d \) and all \( 1 \leq k \leq N \).
\end{enumerate}
\end{theorem}

\begin{proof}
The implication \ref{th:hyperplanes<=>bound-a}\( \IMPLIES \)\ref{th:hyperplanes<=>bound-b} follows from the Erd\H{o}s-Jackson-Mauldin Theorem~\ref{th:EJM2} and Example~\ref{xmp:well-spread}.

\smallskip

The implication \ref{th:hyperplanes<=>bound-b}\( \IMPLIES \)\ref{th:hyperplanes<=>bound-c} is easy.
Let \( \mathbf{u}_1 , \dots , \mathbf{u}_N \) be as in~\ref{th:hyperplanes<=>bound-c}.
Without loss of generality we may assume that \( \mathbf{u}_1 , \dots , \mathbf{u}_{ ( n + 1 ) ( d - 1 ) + 1 } \) are in general position in \( \R^d \), and let \( A_1 , \dots , A_{ ( n + 1 ) ( d - 1 ) + 1 } \) be as in~\ref{th:hyperplanes<=>bound-b}.
Letting \( A_{ ( n + 1 ) ( d - 1 ) + 2} = \dots = A_N = \emptyset \) we have sets as in~\ref{th:hyperplanes<=>bound-c}.

\smallskip

The implication \ref{th:hyperplanes<=>bound-c}\( \IMPLIES \)\ref{th:hyperplanes<=>bound-d} is trivial, as any set of vectors in general position spans the space, and any covering of \( \R^d \) is a covering of \( D \). 

\smallskip

Assume~\ref{th:hyperplanes<=>bound-d} and towards a contradiction suppose that \( 2^{\aleph_0} > \aleph_{ n + 1 } \).
By a translation, we may assume that \( ( - \varepsilon ; \varepsilon )^d \subseteq D \).
The hypotheses of Theorem~\ref{th:hyperplanes=>bound} are satisfied so there is a set \( Z \subseteq D \) that is not contained in \(\bigcup_{ k = 1 }^N A_k \), so the sequence \( A_1 , \dots , A_ N \) does not cover \( D \).
\end{proof}

When \( d = 3 \) Theorem~\ref{th:hyperplanes<=>bound} yields the following results---recall that vectors in \( \R^3 \) are in general position if any three of them are linearly independent.
First we consider the case when the intersection of a plane with a set is finite.

\begin{corollary}\label{cor:hyperplanes=>CH-1}
Let \( D \subseteq \R^3 \) be such that \( \emptyset \neq \Int ( D ) \), and let \( n \geq 1 \).
The following are equivalent:
\begin{enumerate}[label={\upshape (\alph*)}]
\item
 \( 2^{\aleph_0 } \leq \aleph_n \);
\item
for all \( \mathbf{u}_1 , \dots , \mathbf{u}_{ 2 n + 3 } \) vectors in general position in \( \R^3 \), then there are \( A_1 , \dots , A_{ 2 n + 3 } \) covering \( D \) so that every plane orthogonal to \( \mathbf{u}_i \) intersects \( A_i \) in a finite set, for all \( 1 \leq i \leq 2 n + 3 \);
\item
there are non-zero vectors \( \mathbf{u}_1 , \dots , \mathbf{u}_{ 2 n + 4 } \) in \( \R^3\), and sets \( A_1 , \dots , A_{ 2 n + 4 } \) covering \( D \) such that every plane orthogonal to \( \mathbf{u}_i \) intersects \( A_i \) in a finite set, for all \( 1 \leq i \leq 2 n + 4 \).
\end{enumerate}
\end{corollary}

In particular~\cite{Erdos:1994yq,Simms:1997aa} \( \CH \) is equivalent to 
\begin{itemize}
\item
for all \( \mathbf{u}_1 , \dots , \mathbf{u}_5 \) in general position there are \( A_1 , \dots , A_5 \) covering \( \R^3 \) such that any plane orthogonal to \( \mathbf{u}_i \) has finite intersection with \( A_i \),
\item
there are non-zero vectors \( \mathbf{u}_1 , \dots , \mathbf{u}_6 \) and \( A_1 , \dots , A_6 \) covering \( \R^3 \) such that any plane orthogonal to \( \mathbf{u}_i \) has finite intersection with \( A_i \).
\end{itemize}

If the intersections between planes and sets are taken to be countable we have a result analogous to Corollary~\ref{cor:hyperplanes=>CH-1}, with a shift of \( 2 \) in the number of vectors/pieces.

\begin{corollary}
Let \( D \subseteq \R^3 \) be such that \( \emptyset \neq \Int ( D ) \), let \( n \geq 1 \) and let \( N \) be such that \( 2 n + 2 \leq N \).
The following are equivalent:
\begin{enumerate}[label={\upshape (\alph*)}]
\item
 \( 2^{\aleph_0 } \leq \aleph_n \);
\item
for all \( \mathbf{u}_1 , \dots , \mathbf{u}_{ 2 n + 1} \) vectors in general position there are \( A_1 , \dots , A_{ 2 n + 1 } \) covering \( D\) so that every plane orthogonal to \( \mathbf{u}_k \) intersects \( A_k \) in a countable set, for all \( 1 \leq k \leq 2 n + 1 \);
\item
for all vectors \( \mathbf{u}_1 , \dots , \mathbf{u}_{ 2 n + 2 } \) non-zero vectors such that at least \( 2 n + 1 \) of them are in general position, there are \( A_1 , \dots , A_{ 2 n + 2 } \) covering \( D \) such that every plane orthogonal to \( \mathbf{u}_k \) intersects \( A_k \) in a countable set, for all \( 1 \leq k \leq 2 n + 2 \);
\item
there are vectors \( \mathbf{u}_1 , \dots , \mathbf{u}_{ 2 n + 2 } \) and sets \( A_1 , \dots , A_{ 2 n + 2 } \) covering \( D \) such that every plane orthogonal to \( \mathbf{u}_k \) intersects \( A_k \) in a countable set, for all \( 1 \leq k \leq 2 n + 2 \).
\end{enumerate}
\end{corollary}

The tables in Figure~\ref{fig:1} summarize the results proved so far.
The table on the left, for any dimension \( d \) and any \( n \) such that \( 2^{\aleph_0} \leq \aleph_n \), gives \( ( n + 1 ) ( d - 1 ) + 1 \) the minimum number of vectors/pieces of a covering of \( \R^d \) so that each piece has \emph{finite} intersection with any hyperplane orthogonal to the given vector.
The maximum number for such decomposition is \( ( n + 2 ) ( d - 1 ) \), the integer in the square with coordinates \( ( d , n + 1 ) \) decreased by \( 1 \).
The table on the right is similar, but the intersections with the hyperplanes are \emph{countable}.

\begin{figure}
\hfil\begin{tabular}{c|c|c|c|c|c}
\diagbox{\( d \)}{\( n \)} & \( 1 \) & \( 2 \) & \( 3 \) & \( 4 \) & \ldots
\\
\hline
\( 2 \) & \( 3 \) & \( 4 \) & \( 5 \) & \( 6 \) & \ldots 
\\
\hline
\( 3 \) & \( 5 \) & \( 7 \) & \( 9 \) & \( 11 \) & \ldots
\\
\hline
\( 4 \) & \( 7 \) &\( 10 \) & \( 13 \) & \( 16 \) & \ldots 
\\
\hline
\( 5 \) & \( 9 \) & \( 13 \) & \( 16 \) & \( 20 \) & \ldots
\\
\hline
\( \vdots \) & & &&&\( \ddots \)
\end{tabular}
\hfil
\begin{tabular}{c|c|c|c|c|c}
\diagbox{\( d \)}{\( n \)} & \( 1 \) & \( 2 \) & \( 3 \) & \( 4 \) & \ldots
\\
\hline
\( 2 \) & \( 2 \) & \( 3 \) & \( 4 \) & \( 5 \) & \ldots 
\\
\hline
\( 3 \) & \( 3 \) & \( 5 \) & \( 7 \) & \( 9 \) & \ldots
\\
\hline
\( 4 \) & \( 4 \) &\( 7 \) & \( 10 \) & \( 13 \) & \ldots 
\\
\hline
\( 5 \) & \( 5 \) & \( 9 \) & \( 13 \) & \( 17 \) & \ldots
\\
\hline
\( \vdots \) & & &&&\( \ddots \)
\end{tabular}\hfil
 \caption{Tables of the minimum number of pieces of a decomposition of \( \R^d \) equivalent to \( 2^{\aleph_0} \leq \aleph_n \): the one on the left for \emph{finite} intersections, the one on the right for \emph{countable} intersections.}
 \label{fig:1}
\end{figure}

\begin{proof}[Proof of Theorem~\ref{th:hyperplanes=>bound}]
Fix \( \delta \in \On \) and \( d \geq 3 \), and proceed by induction on  \( n \).

The case \( n = 0 \) is Theorem~\ref{th:EJM-Rd}, so assume the result holds for some \( \bar{n} \) towards proving it for \( \bar{n} + 1 \).

Suppose \( 2^{\aleph_0} > \aleph_{ \delta + \bar{n} + 1 } \).
Let \( N = ( \bar{n} + 2 ) ( d - 1 ) + 1 \), and let \( \mathbf{u}_1 , \dots , \mathbf{u}_ N \) and \( A_1 , \dots , A_ N \) be as in the statement.
Given \( \varepsilon > 0 \) we must construct \( Z = Z_{ \bar{n} + 1 , \varepsilon } \) such that 
\[ 
Z \subseteq ( - \varepsilon ; \varepsilon )^d , \quad \card{Z} = \aleph_{ \delta + \bar{n} + 2 } , \quad \FORALL{ \mathbf{p} \in \R^d } \bigl ( \mathbf{p} + Z \nsubseteq \textstyle \bigcup_{k = 1 }^N A_k \bigr ) .
\] 
Let \( \bar{N} = ( \bar{n} + 1 ) ( d - 1 ) + 1 \).
Then \( \mathbf{u}_1 , \dots , \mathbf{u}_{\bar{N} } \) and \( A_1 , \dots , A_{ \bar{N} } \) satisfy the hypotheses of the statement for \( \bar{n} \).
By inductive assumption there is \( \bar{Z} = Z_{ \bar{n} , \varepsilon / 2 } \subseteq ( - \varepsilon / 2 ; \varepsilon / 2 )^d \) of size \( \aleph_{ \delta + \bar{n} + 1 } \) such that 
\begin{equation}\label{eq:th:hyperplanes=>bound-1}
 \FORALL{ \mathbf{p} \in \R^d } \bigl ( \mathbf{p} + \bar{Z} \nsubseteq \textstyle \bigcup_{ k = 1 }^{\bar{N}} A_k \bigr ) .
\end{equation}
As \( N - \bar{N} = d - 1 \), there is a non-zero vector \( \mathbf{v} \) such that 
\begin{equation}\label{th:hyperplanes=>bound-2}
 \bar{N} < k \leq N \IMPLIES \mathbf{v} \boldsymbol{\cdot} \mathbf{u}_k = 0 . 
\end{equation}
The subspace \( V = \setof{ r \mathbf{v}}{ r \in \R } \) is of cardinality \( 2^{\aleph_0 } \geq \aleph_{ \delta + \bar{n} + 2 } \), and \( X = \setof{ r \in \R }{ r \mathbf{v} \in \bar{Z}} \) is of cardinality \( \aleph_{ \delta + \bar{n} + 1 } \), so by Lemma~\ref{lem:setofdifferenceinabeliangroup} we obtain \( S \subseteq ( - \varepsilon / 2 ; \varepsilon / 2 ) \) of size \( \aleph_{ \delta + \bar{n} + 2 } \) such that \( ( S - S ) \cap ( X - X ) = \set{0} \).
Letting \( Y = \setof{ s \mathbf{v} }{ s \in S } \) then \( \card{Y} = \aleph_{ \delta + \bar{n} + 2 } \) and \( Y \subseteq ( - \varepsilon / 2 ; \varepsilon / 2 )^d \) and 
\[
 ( Y - Y ) \cap ( \bar{Z} - \bar{Z} ) = \setLR{\mathbf{0}} .
\]
Let 
\[
 Z \coloneqq Y + \bar{Z}= \textstyle \bigcup_{ \mathbf{y} \in Y} \mathbf{y} + \bar{Z} = \bigcup_{ \mathbf{z} \in \bar{Z} } \mathbf{z} + Y .
\]
Observe that \( Z \subseteq ( - \varepsilon ; \varepsilon )^d \) and is of cardinality \( \aleph_{ \delta + \bar{n} + 2 } \).
We must argue that \( \mathbf{p} + Z \nsubseteq \bigcup_{ k = 1 }^ N A_k \) for all \( \mathbf{p} \in \R^d \).

Towards a contradiction, let \( \hat{\mathbf{p}} \in \R^d \) be such that \( \hat{\mathbf{p}} + Z \subseteq \bigcup_{ k = 1 }^N A_k \).

\begin{claim}\label{cl:th:hyperplanes=>bound-1}
If \( \mathbf{y} , \mathbf{y}' \in Y \) and \( ( \mathbf{y} + \bar{Z} ) \cap ( \mathbf{y}' + \bar{Z} ) \neq \emptyset \) then \( \mathbf{y} = \mathbf{y}' \).
\end{claim}

\begin{proof}
Suppose \( \mathbf{y} + \mathbf{z} = \mathbf{y}' + \mathbf{z}' \) with \( \mathbf{y} , \mathbf{y}' \in Y \) and \( \mathbf{z} , \mathbf{z}' \in \bar{Z} \).
Then \( \mathbf{y} - \mathbf{y} ' = \mathbf{z}' - \mathbf{z} \in ( Y - Y ) \cap ( \bar{Z} - \bar{Z} ) = \set{ \mathbf{0} } \), so \( \mathbf{y} = \mathbf{y}' \).
\end{proof}

Recall that \( Y \subseteq V \) and by~\eqref{th:hyperplanes=>bound-2} \( V \) is contained in \( H_{\mathbf{u}_k } ( \mathbf{0} ) \), for \( \bar{N} < k \leq N \).
By assumption \( \card{ A_k \cap H_{\mathbf{u}_k } ( \mathbf{q} ) } < \aleph_ { \delta } \), so \( \card{ A_k \cap ( \mathbf{q} + Y ) } < \aleph_ { \delta } \), for each \( \mathbf{q} \in \R^d \).
Therefore for any \( \bar{N} < k \leq N \),
\begin{equation}\label{th:hyperplanes=>bound-3}
\card{ A_k \cap ( \hat{\mathbf{p}} + Z ) } = \card{ A_k \cap ( \hat{\mathbf{p}} + Y + \bar{Z} ) } = \cardLR{ \textstyle \bigcup_{ \mathbf{z} \in \bar{Z} } A_k \cap ( \hat{\mathbf{p}} + \mathbf{z} + Y ) } \leq \aleph _{ \delta + \bar{n} + 1 } . 
\end{equation}

\begin{claim}\label{cl:th:hyperplanes=>bound-2}
There is \( \hat{\mathbf{y}} \in Y \) such that \( ( \hat{\mathbf{p}} + \hat{\mathbf{y}} + \bar{Z} ) \cap \bigcup_{ \bar{N} < k \leq N} A_k = \emptyset \).
\end{claim}

\begin{proof}
Towards a contradiction, suppose that for all \( \mathbf{y} \in Y \) there are \( \mathbf{z} ( \mathbf{y} ) \in \bar{Z} \) and \( \bar{N} < k ( \mathbf{y} ) \leq N \) such that \( \hat{\mathbf{p}} + \mathbf{y} + \mathbf{z} ( \mathbf{y} ) \in A_{ k ( \mathbf{y} ) } \).
As 
\[
 \aleph_{ \delta + \bar{n} + 2 } = \card{Y} > \aleph_{ \delta + \bar{n} + 1 } =\bar{Z} > \card{\setof{ k }{ \bar{N} < k \leq N}} 
\] 
there is \( \hat{Y} \subseteq Y \) of size \( \aleph_{ \delta + \bar{n} + 2 } \), and \( \hat{\mathbf{z}} \in \bar{Z} \) and \( \bar{N} < \hat{k} \leq N \) such that \( \mathbf{z} ( \mathbf{y} ) = \hat{\mathbf{z}} \) and \( k ( \mathbf{y} ) = \hat{k} \) for all \( \mathbf{y} \in \hat{Y} \).
Therefore \( \FORALL{ \mathbf{y} \in \hat{Y} } ( \hat{\mathbf{p}} + \mathbf{y} + \hat{\mathbf{z}} \in A_{\hat{k}} ) \).
By Claim~\ref{cl:th:hyperplanes=>bound-1} the map \( \hat{Y} \to A_{\hat{k}} \), \( \mathbf{y} \mapsto \hat{\mathbf{p}} + \mathbf{y} + \hat{\mathbf{z}} \) is injective, so \( \card{ A_{\hat{k}} \cap ( \hat{\mathbf{p}} + Y + \bar{Z} ) } \geq \card{ \hat{Y} } = \aleph _{ \delta + \bar{n} + 2 } \), against~\eqref{th:hyperplanes=>bound-3}.
\end{proof}

Fix \( \hat{\mathbf{y}} \) as in Claim~\ref{cl:th:hyperplanes=>bound-2}.
Then \( ( \hat{\mathbf{p}} + \hat{ \mathbf{y} } ) + \bar{Z} \subseteq A_1 \cup \dots \cup A_{\bar{N}} \) against~\eqref{eq:th:hyperplanes=>bound-1}.
Having reached a contradiction, we conclude that \( \mathbf{p} + Z \nsubseteq \bigcup_{ k = 1 }^ N A_k \) for all \( \mathbf{p} \in \R^d \).
\end{proof}

\begin{remark}\label{rmk:furtherresults}
Theorems~\ref{th:EJM-Rd} and ~\ref{th:hyperplanes=>bound} (together with Theorem~\ref{th:EJM2}) provide a complete solution to Problem~\ref{prob:hyperplanes} when the vectors are in general position, yet some further generalizations are possible.
For the sake of readability we have opted for less generality, and here we would like to mention some of these extensions.
(The proof of these results will appear elsewhere.)

Focusing on \( d = 3 \) and \( \delta = 0 \), the requirement on the size of the intersections in Theorem~\ref{th:EJM-Rd} could be relaxed to
\begin{align*}
\card{ A_i \cap H_{ \mathbf{u}_i } ( \mathbf{p} ) } < \aleph_0 \quad i & = 1 , 2 , 
&
\card{ A_i \cap H_{ \mathbf{u}_i } ( \mathbf{p} ) } \leq \aleph_0 \quad i & = 3 , 4 ,
\end{align*}
strengthening a theorem of Bagemihl~\cite{Bagemihl:1959te} that there are no \( A_1 , A_2 , A_3 \) covering \( \R^3 \) so that all planes orthogonal to \( \mathbf{e}_2 , \mathbf{e}_3 \) have countable intersections with \( A_2 , A_3 \) and all planes orthogonal to \( \mathbf{e}_1 \) have finite intersections with \( A_1 \).
(By Theorem~\ref{th:Sierpinski} and Lemma~\ref{lem:changecoordinatestogetorthogonalplanes} \( \CH \) is equivalent to the fact that the sets \( A_i \cap H_{ \mathbf{u}_i } ( \mathbf{p} ) \) are countable for three distinct \( i \)s, if \( \mathbf{u}_1, \dots , \mathbf{u}_4 \) are in general position.)

Another possible generalization of Theorem~\ref{th:EJM-Rd} is that for any \( \mathbf{u}_1 , \dots , \mathbf{u}_n \in \Span ( \mathbf{e}_1 , \mathbf{e}_2 ) \setminus \setLR{\mathbf{0}} \) and \( \mathbf{v}_1 , \dots , \mathbf{v}_m \in \Span ( \mathbf{e}_1 , \mathbf{e}_3 ) \setminus \setLR{\mathbf{0}} \), there are no \( A_1 , A_2 , A_3 \), \( B_1 , \dots , B_n \), \( C_1 , \dots , C_m \) covering \( \R^3 \) such that each \( A_i \cap H_{\mathbf{e}_i } ( \mathbf{p} ) \), \( B_j \cap H_{\mathbf{u}_j } ( \mathbf{p} ) \), \( C_k \cap H_{ \mathbf{v}_k } ( \mathbf{p} ) \) is finite.
(Clearly \( ( \mathbf{e}_1 , \mathbf{e}_2 , \mathbf{e}_3 ) \) can be replaced with any other basis.)
This shows that Corollary~\ref{cor:hyperplanes=>CH-1} can fail badly, if the vectors are not in general position.

On the other hand, not being in general position does not preclude a positive result.
For example \( \CH \) is equivalent to the existence of \( A_1 , \dots , A_6 \) covering \( \R^3 \) such that each \( A_i \cap H_{\mathbf{u}_i } ( \mathbf{p} ) \) is finite, where \( \mathbf{u}_1 = \mathbf{e}_1 \), \( \mathbf{u}_2 = \mathbf{e}_2 \), \( \mathbf{u}_3 = \mathbf{e}_3 \), \( \mathbf{u}_4 \in \Span ( \mathbf{e}_1 , \mathbf{e}_2 ) \setminus \set{ \mathbf{0}} \), \( \mathbf{u}_5 \in \Span ( \mathbf{e}_1 , \mathbf{e}_4 ) \setminus \set{ \mathbf{0}} \), \( \mathbf{u}_6 \in \Span ( \mathbf{e}_2 , \mathbf{e}_3 ) \setminus \set{ \mathbf{0}} \).
\end{remark}

\section{The main results}\label{sec:The-main-results}

Combining Theorem~\ref{th:transfer} together with the results from Section~\ref{sec:Hyperplane-sections} we are ready to prove the results about sprays in \( \R^d \).

\begin{theorem}\label{th:no4spraysinR3}
Let \( \mathbf{c}_1 , \dots , \mathbf{c}_4 \) be four coplanar points in \( \R^3 \).
There are no \( X_1 , \dots , X_4 \) covering \( \R^3 \) such that \( X_i \) is a spray with centers \( \mathbf{c}_i \).
\end{theorem}

\begin{proof}
Towards a contradiction, suppose \( \R^3 = X_1 \cup \dots \cup X_4 \).

If \( \mathbf{c}_1 , \dots , \mathbf{c}_4 \) belong to the same line \( \ell \), let \( P \) be the plane orthogonal to \( \ell \), and let \( \pi \colon \R^3 \to P \) be the orthogonal projection.
Then \( Y \coloneqq \bigcup_{1 \leq i \leq 4} \pi [ X_i ] \) is a spray of \( P \) centered in \( \pi ( \mathbf{c}_1 ) = \dots = \pi ( \mathbf{c}_4 ) \in P \), and since the \( X_i \) cover \( \R^3 \), then \( Y \) should cover \( P \cong \R^2 \), which is absurd.

Therefore we may assume that \( \mathbf{c}_1 , \dots , \mathbf{c}_4 \) are not collinear, so either \( \mathbf{c}_1 , \mathbf{c}_2 , \mathbf{c}_3 \) are not collinear or \( \mathbf{c}_1 , \mathbf{c}_2 , \mathbf{c}_4 \) are not collinear.
If \( \mathbf{c}_1 , \mathbf{c}_2 , \mathbf{c}_3 \) are not collinear, then by Theorem~\ref{th:transfer} there is an open \( E^3 \subseteq \R^3 \) and there are \( A_1 , \dots , A_4 \) covering \( E^3 \) such that \( A_i \cap H_i ( r ) \) are finite for \( i = 1 , 2 , 3 \), and \( A_4 \cap H_{\mathbf{u}} ( \mathbf{p} ) \) is finite, where \( \mathbf{u} = ( u_1 , u_2 , u_3 ) \neq \mathbf{0} \) is such that \( u_1 ( \mathbf{c}_1 - \mathbf{c}_4 ) + u_2 ( \mathbf{c}_2 - \mathbf{c}_4 ) + u_3 ( \mathbf{c}_3 - \mathbf{c}_4 ) = \mathbf{0} \) and \( x \) and \( \mathbf{p} \) range in \( \R \) and \( \R^3 \), respectively.
The result follows from Corollaries~\ref{cor:EJM-Rd-2}.
Finally if \( \mathbf{c}_0 , \mathbf{c}_1 , \mathbf{c}_3 \) are not collinear, then swapping \( \mathbf{c}_2 \) with \( \mathbf{c}_3 \) we fall back into the previous case.
\end{proof}

Suppose \( \mathbf{c}_1 , \dots , \mathbf{c}_d \) are distinct, well-placed points in \( \R^d \), belonging to some hyperplane \( H \), and recall the vector space \( \mathcal{U} \) of~\eqref{eq:U}.
 
\begin{proposition}\label{prop:frompointstovectors}
Suppose \( n \geq d \) and \( \mathbf{c}_1 , \dots , \mathbf{c}_n \) are well-placed points in \( \R^d \).
For \( k \leq n \) let \( \mathbf{u}_k \in \mathcal{U} (\mathbf{c}_k ) \setminus \set{ \mathbf{0} } \).
Then the vectors \( \mathbf{u}_1 , \dots , \mathbf{u}_n \) are in general position in \( \R^d \).
\end{proposition}

\begin{proof}
Let \( H \) be the hyperplane passing through the points \( \mathbf{c}_i \)s.
As the definition of \( \mathcal{U} ( \mathbf{c}_k ) \) is not affected by a translation of \( H \), we may assume that \( \mathbf{0} \notin H \), so that \( ( \mathbf{c}_1 , \dots , \mathbf{c}_ d ) \) forms a basis of \( \R^d \).
Let \( \mathbf{u}_k = ( u_k^1 , \dots , u_k^d ) \).
We claim that \( \sum_{ i = 1 }^d u_k^i \neq 0 \), for any \( k \leq n \).
Otherwise 
\begin{multline*}
\mathbf{0} = \sum_{ i = 1 }^d u_k ^i ( \mathbf{c}_i - \mathbf{c}_k ) = \sum_{ i = 1 }^d u_k ^i \mathbf{c}_i = u_k^1 \mathbf{c}_1 + \sum_{ i = 2 }^d u_k ^i \mathbf{c}_i 
\\
= - ( \sum_{ i = 2 }^d u_k ^i ) \mathbf{c}_1 + \sum_{ i = 2 }^d u_k ^i \mathbf{c}_i = \sum_{ i = 2 }^d u_k ^i ( \mathbf{c}_i - \mathbf{c} _1 )
\end{multline*}
shows that \( \mathbf{c}_1 , \dots , \mathbf{c}_d \) are not in general position in \( H \), against our assumption.
Scaling of vectors does not affect their general position, so we may assume that \( \sum_{ i = 1 }^d u_k^i = 1 \) for all \( k \leq n \).
Then \( \sum_{ i = 1}^d u_k^i ( \mathbf{c}_i - \mathbf{c}_k ) = \mathbf{0} \) implies that
\[
\mathbf{c}_k = \sum_{ i = 1}^d u_k^i \mathbf{c}_i
\]
that is to say: \( ( u_k^1 , \dots , u_k^d ) \) are the components of the vector \( \mathbf{c}_k \) with respect to the basis \( ( \mathbf{c}_1 , \dots , \mathbf{c}_d ) \).
Therefore the general position of the \( \mathbf{u}_k \)s follows from the general position of the \( \mathbf{c}_k \)s in \( H \).
\end{proof}

The very same proof yields:

\begin{proposition}\label{prop:frompointstovectors2}
Let \( H \) be a hyperplane of \( \R^d \), and suppose that for all \( k \geq 1 \) the points \( \mathbf{c}_k \in H \) are in general position in \( H \).
For all \( k \geq 1 \) let \( \mathbf{u}_k \in \mathcal{U} (\mathbf{c}_k ) \setminus \set{ \mathbf{0} } \).
Then the vectors \( \mathbf{u}_1 , \mathbf{u}_2 , \dots \) are in general position in \( \R^d \).
\end{proposition}

\begin{theorem}\label{th:spraysinRd}
Fix \( d \geq 2 \) and \( n \geq 1 \), and let \( N = ( n + 1 ) ( d - 1) + 1 \) and \( M = ( n + 2 ) ( d - 1 ) \).
The following are equivalent:
\begin{enumerate}[label={\upshape (\alph*)}]
\item\label{th:spraysinRd-a}
\( 2^{\aleph_0 } \leq \aleph_n \).
\item\label{th:spraysinRd-b}
For all distinct, well-placed points \( \mathbf{c}_1 , \dots , \mathbf{c}_N \in \R^d \) there are sprays \( X_1 , \dots , X_N \) covering \( \R^d \) with \( X_i \) centered in \( \mathbf{c}_i \).
\item\label{th:spraysinRd-c}
There are sprays \( X_1 , \dots , X_M \) covering \( \R^d \) with \( X_i \) centered in \( \mathbf{c}_i \) such that \( \mathbf{c}_1 , \dots , \mathbf{c}_N \) are distinct and well-placed.
\end{enumerate}
\end{theorem}

\begin{proof}
\ref{th:spraysinRd-a}\( \IMPLIES \)\ref{th:spraysinRd-b} follows from Theorem~\ref{th:boundoncontinuum=>R^dcoveredwithkappasprays} with \( \delta = 0 \), and \ref{th:spraysinRd-b}\( \IMPLIES \)\ref{th:spraysinRd-c} is trivial---set \( X_{N + 1} = \dots = X_M = \emptyset \) and \( \mathbf{c}_{ N + 1 } , \dots , \mathbf{c}_ M \) any points belonging to the same hyperplane passing through \( \mathbf{c}_1 , \dots , \mathbf{c}_N \).
So it is enough to prove \ref{th:spraysinRd-c}\( \IMPLIES \)\ref{th:spraysinRd-a}.

If \( d = 2 \) then \( N = M = n + 2 \) and \ref{th:spraysinRd-c} says that \( \R^2 \) is the union of \( n + 2 \) sprays with aligned centers, so \( 2^{\aleph_0 } \leq \aleph_n \) follows from~\cite[Theorem 7]{Schmerl:2010nr}.
Therefore we may assume that \( d \geq 3 \).
Towards a contradiction, assume \( 2^{\aleph_0 } \geq \aleph_{ n + 1} \) and suppose that \( \mathbf{c}_1 , \dots , \mathbf{c}_M \in \R^d \) are well-placed, and that \( X_1 , \dots , X_M \) are sprays as above.
For ease of notation, let \( \mathbf{u}_i = \mathbf{e}_i \) for \( 1 \leq i \leq d \).
By Proposition~\ref{prop:frompointstovectors} the vectors \( \mathbf{u}_1 , \dots , \mathbf{u}_N \) are in general position in \( \R^d \), and the map \( \Phi \colon \mathbb{H}^d \to E^d \) transforms the \( X_i \) into a covering \( A_1 ,\dots , A_M \) of \( E^d \) such that every hyperplane orthogonal to \( \mathbf{u}_i \) intersects \( A_i \) in a finite set.
But this contradicts Theorem~\ref{th:hyperplanes<=>bound} with \( \delta = 0 \).
\end{proof}

In particular:

\begin{theorem}\label{th:spraysinR3}
The following are equivalent.
\begin{enumerate}[label={\upshape (\alph*)}]
\item\label{th:spraysinR3-a}
\( \CH \).
\item\label{th:spraysinR3-b}
For all well-placed \( \mathbf{c}_1 , \dots , \mathbf{c}_5 \in \R^3 \) there are sprays \( X_1 , \dots , X_5 \) covering \( \R^3 \) with \( X_i \) centered in \( \mathbf{c}_i \).
\item\label{th:spraysinR3-c}
There are sprays \( X_1 , \dots , X_6 \) covering \( \R^3 \) with \( X_i \) centered in \( \mathbf{c}_i \) such that \( \mathbf{c}_1 , \dots , \mathbf{c}_6 \) are well-placed.
\end{enumerate}
\end{theorem}

More generally:

\begin{theorem}\label{th:spraysinRd-2}
Fix \( d \geq 2 \) and \( n \geq 1 \), and let \( M = ( n + 2 ) ( d - 1 ) \), \( N = ( n + 1 ) ( d - 1) + 1 \), \( \bar{M} = (n + 1 ) ( d - 1 ) = N - 1 \), and \( \bar{N} = n ( d - 1) + 1 \).
The following are equivalent:
\begin{enumerate}[label={\upshape (\alph*)}]
\item\label{th:spraysinRd-2-a}
\( 2^{\aleph_0 } \leq \aleph_n \).
\item\label{th:spraysinRd-2-b}
For all well-placed points \( \mathbf{c}_1 , \dots , \mathbf{c}_N \in \R^d \) there are sprays \( X_1, \dots , X_N \) covering \( \R^d \) with \( X_i \) centered in \( \mathbf{c}_i \).
\item\label{th:spraysinRd-2-c}
There are sprays \( X_1 , \dots , X_M \) covering \( \R^d \) with \( X_i \) centered in \( \mathbf{c}_i \) such that \( \mathbf{c}_1 , \dots , \mathbf{c}_M \) are well-placed.
\item\label{th:spraysinRd-2-d}
For all well-placed points \( \mathbf{c}_1 , \dots , \mathbf{c}_{\bar{N}} \in \R^d \) there are \( \sigma \)-sprays \( X_1, \dots , X_N \) covering \( \R^d \) with \( X_i \) centered in \( \mathbf{c}_i \).
\item\label{th:spraysinRd-2-e}
There are \( \sigma \)-sprays \( X_1 , \dots , X_{\bar{M}} \) covering \( \R^d \) with \( X_i \) centered in \( \mathbf{c}_i \) such that \( \mathbf{c}_1 , \dots , \mathbf{c}_{\bar{M}} \) are well-placed.
\end{enumerate}
\end{theorem}

Proposition~\ref{prop:frompointstovectors2} is the bridge connecting Problem~\ref{prob:hyperplanes} with
\begin{problem}\label{prob:sprays}
Given \( \mathbf{c}_1 , \dots , \mathbf{c}_n \) distinct points of \( H \), a hyperplane of \( \R^d \), what conditions on the cardinality of \( \R \) are equivalent to the existence of \( X_1 , \dots , X_n \) covering \( \R^d \), each \( X_i \) a spray (or \( \sigma \)-spray) centered in \( \mathbf{c}_i \)?
\end{problem}
Theorem~\ref{th:spraysinRd-2} yields a complete solution to Problem~\ref{prob:sprays} when the \( \mathbf{c}_i \)s are well-placed, i.e. in general position in \( H \).
Using the results mentioned in Remark~\ref{rmk:furtherresults} it is possible to distill a few more results on sprays in \( \R^3 \).
Let us mention just two of them.
The first is that given four coplanar points \( \mathbf{c}_1 , \mathbf{c}_2 , \mathbf{c}_3 , \mathbf{c}_4 \) in \( \R^3 \), there exist no \( X_1 , X_2 , X_3 , X_4 \) covering \( \R^3 \) such that \( X_1 , X_2 \) are sprays centered in \( \mathbf{c}_1 , \mathbf{c}_2 \), and \( X_3 , X_4 \) are \( \sigma \)-sprays centered in \( \mathbf{c}_3 , \mathbf{c}_4 \).
The second result is that the six points in \( \R^2 \times \setLR{0} \) of Figure~\ref{fig:2} are not in general position in the plane, and 
\begin{itemize}
\item
no five sprays centered in these points can cover \( \R^3 \), but 
\item
\( \CH \) is equivalent to the existence of six sprays, centered in these points, covering \( \R^3 \).
\end{itemize}

\begin{figure}
\[
\begin{tikzpicture}
\draw[very thin, dashed] (0,0)--(2 , 0)-- (1, 1.73205)--cycle;
\filldraw (0,0) circle [radius=1.5pt] {};
\filldraw (1,0) circle [radius=1.5pt] {};
\filldraw (2,0) circle [radius=1.5pt] {};
\filldraw (1,1.73205) circle [radius=1.5pt] {};
\filldraw (0.5,0.866) circle [radius=1.5pt] {};
\filldraw (1.5,0.866) circle [radius=1.5pt] {};
\end{tikzpicture}
\]
\caption{}\label{fig:2}
\end{figure}

\subsection{Covering the space with infinitely many sprays}
We have seen how covering the space with sprays with well-placed centers is equivalent to giving an upper bound for the size of the continuum, the larger the number of sprays, the weaker the bound.
Next we show that, irrespective of the size of the continuum, the space can be covered with \( \aleph_0 \)-many sprays with well-placed centers.

A \markdef{drizzle} in \( \R^d \) with center \( \mathbf{c} \) is a \( X \subseteq \R^d \) such that any sphere centered in \( \mathbf{c} \) intersects \( X \) in at most one point---thus a drizzle is a very sparse spray.

\begin{theorem}\label{th:unionofcountablymanysprays}
If \( \setof{ \mathbf{c}_n }{ n \geq 1 } \) are well-placed points in \( \R^d \), with \( d \geq 2 \), then there are \( X_n \) covering \( \R^d \) such that \( X_n \) is a drizzle centered in \( \mathbf{c}_n \).
\end{theorem}

\begin{proof}
Without loss of generality we may assume that the \( \mathbf{c}_n \)s belong to \( \R^{ d - 1 } \times \setLR{0} \).
Pick \( \mathbf{u}_n \in \mathcal{U} ( \mathbf{c}_n ) \setminus \set{ \mathbf{0} } \), where \( \mathcal{U} ( \mathbf{q} ) = \mathcal{U}_{ \mathbf{c}_1 , \dots , \mathbf{c}_d } ( \mathbf{q} ) \).
By Proposition~\ref{prop:frompointstovectors2} the vectors \( \mathbf{u}_n \) (\( n \geq 1 \)) are in general position in \( \R^d \).
By~\cite{Davies:1974vn} there are sets \( A_k \) (\( k \geq 1 \)) such that for all \( \mathbf{p} \in \R^d \), \( H_{ \mathbf{u}_k } ( \mathbf{p} ) \cap A_k \) has at most one point, and such that \( \R^d = \bigcup_{n \geq 1 } A_{ 2 n } = \bigcup_{n \geq 0 } A_{ 2 n + 1 } \).
The map \( \Phi \colon \mathbb{H}^d \to E^d \) of~\eqref{eq:Phi} can be extended to the closures of \( \mathbb{H}^d \) and \( E^d \), so we can assume that \( X_{ 2 n } \coloneqq \Phi^{-1} [ A_{ 2 n } ] \) is a subset of \( \Cl ( \mathbb{H}^d ) = \R^{ d - 1} \times [ 0 ; + \infty ) \), and that \( \bigcup_{ n \geq 1 } X_{ 2 n } = \Cl ( \mathbb{H}^d ) \).
Letting \( \tau \colon \R^d \to \R^d \) be the reflection with respect to the hyperplane \( \R^{ d - 1 } \times \set{0} \), let \( X_{ 2 n + 1 } = \tau [ \Phi^{-1} [ A_{ 2 n + 1 } ] ] \).
Then \( \bigcup_{n \geq 0 } X_{ 2 n + 1 } = \tau[ \mathbb{H}^d ] = \R^{ d - 1 } \times ( - \infty ; 0 ) \).
Therefore \( \R^d = \bigcup_{n } X_n \), and by construction each \( X_n \) is a drizzle centered in \( \mathbf{c}_n \). 
\end{proof}

\printbibliography

\end{document}